\documentclass[11pt,letterpaper]{article}

\pdfoutput=1

\usepackage[utf8]{inputenc}
\usepackage{amsmath,amssymb,amsthm,mathtools}
\usepackage{aliascnt}
\usepackage[
  bibliosources=references.bib,
  probability=texdefault,
  theorems=false
]{pomegranate}
\allowdisplaybreaks

\theoremstyle{plain}
\newtheorem{theorem}{Theorem}[section]
\newaliascnt{lemma}{theorem}
\newtheorem{lemma}[lemma]{Lemma}
\aliascntresetthe{lemma}
\newaliascnt{proposition}{theorem}
\newtheorem{proposition}[proposition]{Proposition}
\aliascntresetthe{proposition}
\newaliascnt{corollary}{theorem}
\newtheorem{corollary}[corollary]{Corollary}
\aliascntresetthe{corollary}
\newaliascnt{conjecture}{theorem}
\newtheorem{conjecture}[conjecture]{Conjecture}
\aliascntresetthe{conjecture}
\newaliascnt{question}{theorem}
\newtheorem{question}[question]{Question}
\aliascntresetthe{question}
\theoremstyle{definition}
\newaliascnt{definition}{theorem}
\newtheorem{definition}[definition]{Definition}
\aliascntresetthe{definition}
\theoremstyle{remark}
\newaliascnt{remark}{theorem}
\newtheorem{remark}[remark]{Remark}
\aliascntresetthe{remark}

\newcommand{\E}{\mathbb E}
\newcommand{\PP}{\mathbb P}
\newcommand{\LE}{\mathcal E}
\newcommand{\Max}{\operatorname{Max}}
\newcommand{\Min}{\operatorname{Min}}
\newcommand{\gap}{\operatorname{gap}}
\newcommand{\width}{\mathrm{w}}
\newcommand{\eps}{\varepsilon}

\title{On the Gap of Finite Posets}
\author[1]{Alireza Haqi}
\affil[1]{Stanford University,
  \href{mailto:ahaqi@stanford.edu}{\texttt{ahaqi@stanford.edu}}}
\date{}
\hypersetup{
  pdftitle={On the Gap of Finite Posets},
  pdfauthor={Alireza Haqi}
}

\begin{document}

\maketitle

\begin{abstract}
Let $P$ be a finite nonempty poset with $n$ elements, let
$f:P\to\{1,\ldots,n\}$ be a uniformly random order-preserving bijection,
and put $h_P(x)=\E[f(x)]$.  \textcite{AiresKahn} introduced $\gap(P)$ as
the largest difference
between consecutive values in the ordered list consisting of $0$, $n+1$,
and all the expected ranks $h_P(x)$.
Write $\width(P)$ for the largest size of a pairwise incomparable subset.
We prove three results.  First, we prove a weighted strengthening of an
ideal inequality conjectured by Kahn and obtain the explicit gap--width
bound
\[
  \gap(P)\le 2\width(P)-1.
\]
Second, for every $L>0$ we construct a width-two poset such that the
expected-rank list of every maximal chain has a gap of at least $L$, with
$0$ and $|P|+1$ added as endpoints.  Finally, for every
$r\in\mathbb N$, we
construct a poset $P_r$ for which the relative order induced on every nonempty selected set
$X$ has base-two entropy below $3|X|$, while
$\gap(P_r)\ge(3/2)^r$.  Thus the gap can be arbitrarily large while the
induced order on every selected set has relatively small entropy.  The key ideas behind all
three results were found by ChatGPT 5.6 Sol.
\end{abstract}

\section{Introduction}\label{sec:introduction}

Let $P$ be a finite nonempty poset with $n$ elements.  A linear extension
is an order-preserving bijection
$f:P\to\{1,\ldots,n\}$.  Choose $f$ uniformly and write
\[
  h_P(x)=\E[f(x)]
\]
for the expected rank of $x$.  If these expected ranks, with repetition,
are $s_1\le\cdots\le s_n$, then, using the definition introduced by
\textcite{AiresKahn}, put
\[
 \gap(P)=\max\{s_1,s_2-s_1,\ldots,s_n-s_{n-1},n+1-s_n\}.
\]
Thus $\gap(P)$ is the largest interval missed by the expected ranks,
including the intervals next to $0$ and $n+1$.  The width $\width(P)$ is
the largest size of a pairwise incomparable subset.

An \emph{ideal} is a downward-closed subset.  The following conjecture is
due to Kahn and gives strong control of the expected ranks carried by any
ideal.

\begin{conjecture}[Kahn's ideal conjecture]\label{conj:kahn-ideal}
Let $D$ be a nonempty ideal of $P$, and let $A$ be the set of maximal
elements of $D$.  Then
\[
  \max_{a\in A} h_P(a)\ge |D|-|A|+1.
\]
\end{conjecture}

The case in which $D$ has at most two maximal elements first appeared in
\textcite[Conjecture~6.5]{BrightwellTrotter}, where it was attributed to
Kahn; the following paragraph also notes, without formally stating, Kahn's
general conjecture.
\textcite[Conjecture~2.2]{BiroTrotter} restated this case.
Kahn's conjecture partly motivated both of these works.
\textcite[Conjecture~2.11]{AiresKahn} later stated the
arbitrary-ideal form explicitly.

As \textcite{AiresKahn} observed after their Conjecture~2.11,
\cref{conj:kahn-ideal} would imply the stronger bound
\[
 \gap(P)\le2\width(P)-1.
\]
We prove a weighted strengthening of \cref{conj:kahn-ideal} in
\cref{thm:weighted}, and hence establish this bound.

There is a second, more local way to ask whether width controls gaps.

\begin{definition}[Chain gap]
\label{def:chain-gap}
Let $P$ be a finite nonempty poset and let
$C=\{y_1<\cdots<y_m\}$ be a nonempty chain in $P$.  The \emph{chain gap}
of $C$ in $P$, denoted $\gap_P(C)$, is the largest difference between
consecutive terms in the increasing list
\[
  0,\ h_P(y_1),\ldots,h_P(y_m),\ |P|+1.
\]
\end{definition}

\textcite[Question~12.2]{AiresKahn} ask the following.

\begin{question}[Chain-gap question]\label{question:chain-gap}
Is there a universal constant $K$ such that every finite nonempty poset $P$
contains a nonempty chain $C$ satisfying
\[
  \gap_P(C)\le K\width(P)?
\]
\end{question}

We answer this question negatively: width two does not prevent
the gap of every chain from becoming arbitrarily large.

The third problem concerns the entropy retained after selecting only
some elements.

\begin{definition}[Hereditary entropy per point]
\label{def:hereditary-entropy}
For a finite nonempty poset $P$, let $\boldsymbol\sigma_P$ be a uniformly
random linear extension, viewed as an ordering.  For
$\varnothing\ne X\subseteq P$, let $\boldsymbol\sigma_P|_X$ be the relative
order induced on $X$.  With $H$ denoting Shannon entropy in bits, define
\[
 \tau(P)=\max_{\varnothing\ne X\subseteq P}
 \frac{H(\boldsymbol\sigma_P|_X)}{|X|}.
\]
The quantity $\tau(P)$ is the \emph{hereditary entropy per point} of $P$.
\end{definition}

\textcite[Conjecture~2.5]{AiresKahn} conjectured that the following question
has a positive answer.

\begin{question}[Gap--entropy question]\label{question:gap-entropy}
Must every sequence $(P_j)$ of finite nonempty posets satisfying
$\gap(P_j)\to\infty$ also satisfy $\tau(P_j)\to\infty$?
\end{question}

We answer this question negatively in \cref{thm:c25-main}: we construct
posets $P_r$ such that $\gap(P_r)\to\infty$ while $\tau(P_r)<3$ for every
$r$.

\subsection{Main results and proof ideas}

For any poset $Q$, let $e(Q)$ be its number of linear extensions.  Our first
result, \cref{thm:weighted},
says that if $D$ is a nonempty ideal of $P$ and $A$ is its set of maximal
elements, then
\[
  \sum_{a\in A}\frac{e(D-a)}{e(D)}h_P(a)
  \ge |D|-|A|+1.
\]
Here $D-a$ means that $a$ is removed from $D$.  The coefficients form a
probability distribution, so this weighted statement implies the
unweighted \cref{conj:kahn-ideal}.  Applying it below and above a threshold
in the expected-rank list gives
\[
  \gap(P)\le2\width(P)-1
\]
as stated in \cref{cor:gap}.

Our second result, \cref{thm:q122-main}, shows that even a width-two poset
can have arbitrarily large expected-rank gaps along every chain.  We use two
$N$-element chains $A_1<\cdots<A_N$ and $B_1<\cdots<B_N$, and impose the
cross-relations $B_j<A_i$ exactly when $j\le\ell_i$, where
\[
 \ell_i=\left\lfloor
 N\left(\frac{i}{N+1}\right)^{1/r}\right\rfloor .
\]
When we represent $A$-elements by horizontal steps and $B$-elements by vertical
steps, the linear extensions are precisely the monotone lattice paths from
$(0,0)$ to $(N,N)$ whose $i$th horizontal step has height at least
$\ell_i$.  Every
maximal chain is either the full $B$-chain or a chain that starts with
$B$ and switches once to $A$.
For the full $B$-chain, the gap following $B_N$ is determined by the
expected number of horizontal steps after the final vertical step, and
simple counting shows that this number is large.  For a
chain that switches from $B$ to $A$ at $i$, the gap between $B_{\ell_i}$ and $A_i$
is determined by how far the path lies above the required height there
and by how many earlier horizontal steps have already reached that height.
We show that this gap is large for every $i$ as well.

The gaps of chains
that switch from $B$ to $A$ are treated according to how quickly the required heights increase.
A rapid increase forces the path far above the required height, while a
slow increase creates a long stretch at the same height and therefore many
earlier horizontal steps at the switch height.  In the intermediate case,
the required heights remain close to a straight line on a long interval,
and a midpoint estimate shows that the path typically rises above that
line.  Choosing the parameters appropriately makes every chain gap exceed
any prescribed value.

Our third result, \cref{thm:c25-main}, shows that the global gap can grow
arbitrarily large while hereditary entropy remains bounded.  Set $P_0=C_2$.
At stage $r$, place a chain of length $L_r$ below and another above
$P_{r-1}$, obtaining
\[
 R_r=C_{L_r}\oplus P_{r-1}\oplus C_{L_r},
 \qquad P_r=R_r\sqcup R_r,
\]
where the two copies in the parallel sum are incomparable.  The two added
chains preserve the gap.  If $m_r=|R_r|$, taking two parallel
copies multiplies every expected-rank spacing, including the endpoint
spacings, by $(2m_r+1)/(m_r+1)$.  Consequently
\[
 \gap(P_r)=\prod_{s=1}^r\frac{2m_s+1}{m_s+1}
 \ge(3/2)^r.
\]

The choice of the added chain lengths is used for the entropy bound.  Fix a
selected set $X\subseteq P_r$, and let $A_0,A_1$ be its parts inside the
two copies of $P_{r-1}$ and let $t$ count its selected elements from the
added chains.
When $L_r$ is large, the two copies of $P_{r-1}$ occupy disjoint blocks in
the random merge, except on an event of probability tending to zero
uniformly over $X$.  On the complementary event, each nonempty block can
be contracted to one marker after its internal restricted order is fixed.
The remaining entropy is at most $2t+c$, where $c=1$ if both $A_0$ and
$A_1$ are nonempty and $c=0$ otherwise.  The exceptional events contribute
errors, and we choose the added chain lengths so that the sum of these
errors is small.  Iterating the resulting one-step inequality gives, for
every nonempty $X\subseteq P_r$,
\[
 H(\boldsymbol\sigma_{P_r}|_X)
 <3|X|,
\]
and hence $\tau(P_r)<3$.

Together, the results distinguish three notions.  Width controls gaps in
the full expected-rank list.  Nevertheless, no individual chain need meet
that list densely, and bounded entropy on every restriction does not
prevent arbitrarily large global gaps.

\subsection{Proof organization}

\Cref{sec:ideal-inequality} proves the weighted ideal inequality and the
gap--width bound.  \Cref{sec:chain-counterexample} constructs the
width-two example and analyzes its maximal chains in three slope ranges.
\Cref{sec:entropy-counterexample} gives the recursive entropy example and
proves its gap and entropy bounds.

\section{Related work}\label{sec:related-work}

The questions in this paper belong to the broader study of balance in
uniformly random linear extensions, surveyed by
\textcite{BrightwellBalancedPairs}.  For distinct elements $x,y\in P$, set
\[
 p_P(x\prec y)=\PP\bigl(f(x)<f(y)\bigr),
 \qquad
 \delta(P)=\max_{x\ne y}
 \min\{p_P(x\prec y),p_P(y\prec x)\}.
\]
The classical $1/3$--$2/3$ conjecture, first posed by
\textcite{Kislitsyn} and later independently by \textcite{Fredman} and
\textcite{Linial}, asserts that $\delta(P)\ge 1/3$ whenever $P$ is not a
chain.  The general theorem of \textcite{KahnSaks} gives
$\delta(P)\ge 3/11$, and \textcite{BrightwellFelsnerTrotter} improved this to
\[
 \delta(P)\ge \frac{5-\sqrt5}{10}.
\]
Further positive cases of the $1/3$--$2/3$ conjecture, including several
families of lattices and posets arising from Young diagrams, were proved by
\textcite{OlsonSagan}.

\textcite{KahnSaks} proposed the asymptotic conjecture that
$\delta(P)\to1/2$ uniformly as $\width(P)\to\infty$.  Thus the conjecture
predicts the existence of some pair whose two possible orders are nearly
equally likely; it does not concern every pair.
\textcite[Theorems~1.4(c) and~1.7(a)]{AiresKahn} proved this conclusion when
$\width(P)=\Omega(|P|)$, and proved $\delta(P)\ge 1/e-o(1)$ when
$\width(P)=\omega(\sqrt{|P|})$.

Expected ranks give a complementary one-element statistic: they are the
first moments of the rank distributions of individual elements.  For a fixed
element, \textcite{StanleyHeightSequences} proved that the numbers of linear
extensions placing it at successive ranks form a log-concave sequence.
Correlation inequalities for these distributions were studied by
\textcite{Fishburn} and \textcite{BrightwellTrotter}.  The case of Kahn's
ideal conjecture in which the ideal has at most two maximal elements first
appeared in
\textcite[Conjecture~6.5]{BrightwellTrotter}, where it was attributed to
Kahn; the following paragraph notes that Kahn also proposed the general
conjecture, without formally stating it.
\textcite[Conjecture~2.2]{BiroTrotter} later restated this case.
\textcite[Conjecture~2.11]{AiresKahn} stated
Kahn's arbitrary-ideal conjecture explicitly and, separately, introduced
$\gap(P)$ and formulated the global gap--width conjecture as their
Conjecture~12.1.  The weighted ideal inequality in
\cref{sec:ideal-inequality} proves Kahn's conjecture and, through the
deduction of Aires and Kahn, yields the explicit bound
$\gap(P)\le 2\width(P)-1$.

Width-two posets have a particularly concrete description.  After the poset
is partitioned into two chains, a linear extension is encoded by a monotone
lattice path, with the cross-comparabilities determining the region in which
the path may lie; \textcite{StanleyWidthTwo} gives an equivalent formulation
through order ideals of an associated skew Young diagram.  \textcite{Linial}
proved the $1/3$--$2/3$ conjecture for width two, \textcite{Aigner}
characterized the width-two equality cases, and \textcite{Sah} later refined
the width-two balance theory using path counting.  The cross-product
conjecture was formulated by \textcite{BrightwellFelsnerTrotter}.
\textcite{ChanPakPanovaCPC,ChanPakPanovaKS} used lattice-path injections to
prove its width-two form and extensions of the Kahn--Saks log-concavity
inequality; \textcite{ChanPakPanovaCPCGeneral} obtained further estimates for
arbitrary posets.  Our width-two construction uses the same encoding for a
different purpose: the relevant quantity is the expected clearance between a
random path and a curved lower boundary, which becomes a spacing between
expected ranks.  Its middle-range estimate uses positive-walk estimates of
\textcite{OttVelenik}.

The chain-gap question posed by \textcite[Question~12.2]{AiresKahn} asks
whether every poset contains a chain whose expected ranks have maximum spacing
$O(\width(P))$.  This is genuinely different from the global
gap--width bound: the examples in
\cref{sec:chain-counterexample} show that even at width two, the global
gap--width bound need not have a width-controlled chain witness.

Entropy also appears in sorting under partial information:
\textcite{KahnKim} used graph entropy to identify an unknown linear extension
with $O(\bigl(\log e(P)\bigr))$ comparisons, where $e(P)$ is the number of
linear extensions.  The hereditary entropy considered here is a different
statistic.  \textcite[Conjecture~2.5]{AiresKahn} conjectured that an unbounded
gap forces unbounded hereditary entropy per point.  This statistic measures
multi-element randomness, rather than the ordering of one pair or the
expected position of one element.  The construction in
\cref{sec:entropy-counterexample} shows that bounded hereditary entropy does
not force a bounded global gap.

\section{Preliminaries and notation}\label{sec:preliminaries}

All ambient posets to which $h_P$, $\gap(P)$, or $\tau(P)$ is applied are
finite and nonempty; auxiliary induced subposets may be empty.  For any
finite poset $Q$, write $\LE(Q)$ for its set of linear extensions and
$e(Q)=|\LE(Q)|$.  The empty poset has one linear extension, the empty
ordering, so $e(\varnothing)=1$.

We use the expected-rank, global-gap, and width notation introduced in
\cref{sec:introduction}.  The chain gap and hereditary entropy per point are
defined in \cref{def:chain-gap,def:hereditary-entropy}, respectively.  We
extend the restriction notation to $X=\varnothing$ by taking
$\boldsymbol\sigma_P|_X$ to be the deterministic empty ordering.

For a finite poset $Q$, let $\Max Q$ and $\Min Q$ denote its sets of
maximal and minimal elements, respectively.  For an induced subposet $Q$
and $x\in Q$, the notation $Q-x$ means the induced subposet on
$Q\setminus\{x\}$.  We write $J(P)$ for the distributive lattice of ideals
of $P$, ordered by inclusion.

\section{Gap--width bound}
\label{sec:ideal-inequality}

We begin by proving a weighted strengthening of Kahn's ideal conjecture.
The maximal elements of every ideal carry enough expected rank, in a
weighted sense, to control every gap in the expected-rank list and thereby
prove the gap--width bound.

\begin{theorem}[Weighted ideal inequality]\label{thm:weighted}
Let $P$ be a finite nonempty poset, let $D$ be a nonempty ideal of $P$,
and put $A=\Max D$.  Then
\begin{equation}\label{eq:weighted}
 \sum_{a\in A}\frac{e(D-a)}{e(D)}h_P(a)
 \ge |D|-|A|+1.
\end{equation}
In particular,
\begin{equation}\label{eq:ideal}
 \max_{a\in\Max D}h_P(a)
 \ge |D|-|\Max D|+1.
\end{equation}
\end{theorem}

The unweighted conclusion \eqref{eq:ideal} is equivalent to Kahn's ideal
conjecture, stated explicitly by \textcite[Conjecture~2.11]{AiresKahn} and
attributed there to Kahn.  As they observe immediately afterward, Kahn's
conjecture implies the following corollary.

\begin{corollary}\label{cor:gap}
Every finite nonempty poset $P$ satisfies
\begin{equation}\label{eq:gap-bound}
 \gap(P)\le 2\width(P)-1.
\end{equation}
In particular, this proves the gap--width conjecture
\cite[Conjecture~12.1]{AiresKahn}.
\end{corollary}

\subsection{Last-color monotonicity}

We use the ideal form of Fishburn's inequality.  Its usual formulation is
for filters; the following equivalent version follows by reversing the
order.

\begin{theorem}[\textcite{Fishburn}]\label{thm:fishburn}
If $K$ and $L$ are ideals of a finite poset, then
\begin{equation}\label{eq:fishburn}
 \frac{e(K\cup L)e(K\cap L)}{e(K)e(L)}
 \ge
 \frac{|K\cup L|!\,|K\cap L|!}{|K|!\,|L|!}.
\end{equation}
\end{theorem}

Give the elements of a poset arbitrary colors.  For a nonempty poset $Q$
and a set of colors $\mathcal S$, let $p_Q(\mathcal S)$ denote the
probability that the final element of a uniform linear extension of $Q$
has a color in $\mathcal S$.

\begin{lemma}[Last-color monotonicity]\label{lem:last-color}
Let $c\in\Max Q$, put $R=Q-c$, and suppose $|R|=m\ge1$.  Then
\begin{equation}\label{eq:last-color}
\begin{aligned}
 p_Q(\mathcal S)&\ge p_R(\mathcal S)
 &&\text{if the color of $c$ belongs to $\mathcal S$},\\
 p_R(\mathcal S)&\ge p_Q(\mathcal S)
 &&\text{if the color of $c$ lies outside $\mathcal S$}.
\end{aligned}
\end{equation}
\end{lemma}

\begin{proof}
Fix $b\in\Max R$.  If $b<c$, then its last-place probability in $Q$ is
zero.  Otherwise Fishburn's inequality, applied to the ideals $Q-b$ and
$Q-c=R$, gives
\[
 e(Q)e(R-b)
 \ge \frac{m+1}{m}e(Q-b)e(R).
\]
Consequently,
\[
 \frac{e(Q-b)}{e(Q)}
 \le \frac{m}{m+1}\frac{e(R-b)}{e(R)}
 \le \frac{e(R-b)}{e(R)}.
\]
No nonmaximal element of $R$ becomes maximal when $c$ is added.  If the
color of $c$ belongs to $\mathcal S$, summing the displayed comparison
over maximal $b$ whose colors lie outside $\mathcal S$ proves the first
inequality in \eqref{eq:last-color}.  Applying it to the complementary
set of colors proves the second.
\end{proof}

\subsection{A pointwise inequality on the ideal lattice}

Fix a finite nonempty poset $P$ and a nonempty ideal $D$ of $P$, and set
\begin{equation}\label{eq:notation}
 A=\Max D,
 \qquad U=P\setminus D,
 \qquad \lambda_a=e(D-a),
 \qquad \Lambda=e(D)=\sum_{a\in A}\lambda_a.
\end{equation}
For $B\subseteq A$, write
\begin{equation}\label{eq:lambda-subset}
 \lambda(B)=\sum_{a\in B}\lambda_a.
\end{equation}

\begin{proposition}\label{prop:pointwise}
For every ideal $I$ of $P$,
\begin{equation}\label{eq:pointwise}
 e(I)\sum_{a\in A\cap I}e(D-a)
 \le
 e(D)\sum_{x\in\Max I\cap(A\cup U)}e(I-x).
\end{equation}
\end{proposition}

\begin{proof}
Fix $I$.  If $D\cap I=\varnothing$, then $A\cap I=\varnothing$ and
\eqref{eq:pointwise} is immediate.  Hence assume $D\cap I\ne\varnothing$.
Color the elements of
\begin{equation}\label{eq:good-set}
 \mathcal G_I=(A\cap I)\cup(U\cap I)
\end{equation}
good; color every other element bad.  The final element of an extension
of $D$ lies in $A$, and thus
\begin{equation}\label{eq:pD}
 p_D(\{\mathrm{good}\})
 =\frac{\sum_{a\in A\cap I}e(D-a)}{e(D)}.
\end{equation}

There is a downward path in the ideal lattice from $D$ to $D\cap I$
which deletes elements of $D\setminus I$, all of which are bad.  There is
then an upward ideal-lattice path from $D\cap I$ to $I$ which adds elements
of
\[
 I\setminus D=U\cap I;
\]
all of them are good.  At each step, the changed element is maximal in
the larger ideal, since otherwise deleting it would not leave an ideal.
Lemma~\ref{lem:last-color}, in its deletion and addition forms, therefore
gives the intermediate inequalities
\begin{equation}\label{eq:monotone-path}
 p_I(\{\mathrm{good}\})
 \ge p_{D\cap I}(\{\mathrm{good}\})
 \ge p_D(\{\mathrm{good}\}).
\end{equation}
On the other hand,
\begin{equation}\label{eq:pI}
 p_I(\{\mathrm{good}\})
 =\frac{\sum_{x\in\Max I\cap(A\cup U)}e(I-x)}{e(I)}.
\end{equation}
Combining \eqref{eq:pD}--\eqref{eq:pI} proves
\eqref{eq:pointwise}.
\end{proof}

\subsection{Averaging over ideals}

Write $n=|P|$.  Let $J(P)$ be the set of ideals of $P$.  For
$I\in J(P)$, let $G(I)$ be the number of ways to continue an ordered
prefix whose underlying set is $I$ to a full extension of $P$.  To conclude
the proof of \cref{thm:weighted}, we need the following two identities.

\begin{lemma}\label{lem:prefix-tail}
For every $x\in P$,
\begin{equation}\label{eq:prefix-tail}
 \frac{1}{e(P)}\sum_{I\in J(P)}G(I)e(I)
 \mathbf 1_{\{x\in I\}}
 =n+1-h_P(x).
\end{equation}
\end{lemma}

\begin{proof}
The product $e(I)G(I)$ is the number of full extensions whose first
$|I|$ elements have underlying set $I$.  Thus the numerator on the
left of \eqref{eq:prefix-tail} counts pairs $(f,k)$ with
$f\in\LE(P)$ and $f(x)\le k\le n$.  For fixed $f$ there are
$n+1-f(x)$ such values of $k$.  Dividing by $e(P)$ gives
$\E[n+1-f(x)]=n+1-h_P(x)$.
\end{proof}

\begin{lemma}\label{lem:edge-normalization}
For every $x\in P$,
\begin{equation}\label{eq:edge-normalization}
 \frac{1}{e(P)}
 \sum_{\substack{I\in J(P)\\x\in\Max I}}G(I)e(I-x)=1.
\end{equation}
\end{lemma}

\begin{proof}
For $x\in\Max I$, the product $e(I-x)G(I)$ counts the full extensions
whose ideal-lattice path traverses the edge $I-x\to I$, placing $x$ at
that step.  Every full extension traverses exactly one edge which places
$x$.  The sum in \eqref{eq:edge-normalization} is therefore $e(P)$.
\end{proof}

\begin{proof}[Proof of Theorem~\ref{thm:weighted}]
Put
\[
 d=|D|,\qquad r=|A|.
\]
Since $\sum_{a\in A}\lambda_a=\Lambda$ and
$|A\cup U|=n-d+r$, define
\begin{align*}
 W
 &:=\sum_{a\in A}\lambda_a\bigl(d-h_P(a)\bigr)
   -(r-1)\Lambda\\
 &=\sum_{a\in A}\lambda_a\bigl(n+1-h_P(a)\bigr)
   -\Lambda|A\cup U|.
\end{align*}
Using
$\lambda(A\cap I)=\sum_{a\in A}\lambda_a\mathbf 1_{\{a\in I\}}$
and applying Lemma~\ref{lem:prefix-tail} rewrites the first term in the
last line.  Summing Lemma~\ref{lem:edge-normalization} over
$x\in A\cup U$ and interchanging the sums rewrites the second term.
Therefore,
\begin{equation}\label{eq:flow-sum}
 W=\frac{1}{e(P)}\sum_{I\in J(P)}G(I)
 \left[
 e(I)\lambda(A\cap I)
 -\Lambda\sum_{x\in\Max I\cap(A\cup U)}e(I-x)
 \right].
\end{equation}
Every bracket in \eqref{eq:flow-sum} is nonpositive by
Proposition~\ref{prop:pointwise}, so $W\le0$.  Therefore,
\[
 \sum_{a\in A}\lambda_a\bigl(d-h_P(a)\bigr)
 \le (r-1)\Lambda.
\]
Dividing by $\Lambda=e(D)$ and rearranging gives
\eqref{eq:weighted}.  The numbers $\lambda_a/\Lambda$ are positive and
sum to one, so \eqref{eq:ideal} follows.
\end{proof}

\begin{corollary}\label{cor:isolated-balance}
Let $P$ be a finite poset with $n=|P|$.  For distinct $u,v\in P$, write
\[
 \delta_P(u,v)=\min\{p_P(u\prec v),p_P(v\prec u)\}.
\]
As observed by \textcite[Equation~(63)]{AiresKahn}, if $P$ contains an
element incomparable with every other element, then, as $\width(P)\to\infty$,
there are distinct $u,v\in P$ such that
\[
 \delta_P(u,v)=\frac12-o_n(1),
 \qquad o_n(1)\longrightarrow0.
\]
\end{corollary}

\begin{proof}
Write $n=|P|$ and $w=\width(P)$, and let $x$ be incomparable with every
other element.  If the conclusion failed, then
\textcite[Theorem~1.4(c)]{AiresKahn} would force $w=o(n)$ along a
counterexample subsequence.  Since $h_P(x)=(n+1)/2$, \cref{cor:gap}
gives some $y\ne x$ such that
\[
 \left|h_P(y)-\frac{n+1}{2}\right|
 \le \gap(P)\le 2w-1=o(n).
\]
Inserting $x$ uniformly into an extension of $P-x$ gives
$p_P(x\prec y)=h_P(y)/(n+1)=1/2+o(1)$.  Hence
$\delta_P(x,y)=1/2-o_n(1)$ for some nonnegative $o_n(1)\to0$, a
contradiction.
\end{proof}

\section{Chain-gap counterexample}
\label{sec:chain-counterexample}

In this section we address \cref{question:chain-gap} by constructing a counterexample where every maximal chain has a large gap while the poset itself has width two.

\begin{theorem}\label{thm:q122-main}
For every $L>0$ there is a finite poset $P$ of width two such that
\[
  \gap_P(C)\ge L
\]
for every nonempty chain $C\subseteq P$.
\end{theorem}

\subsection{The counterexample construction}

\begin{definition}
\label{def:q122-poset}
Fix an integer $r\ge2$, and choose $N$ sufficiently large depending on
$r$.  Set
\begin{equation}
 \varphi(x)=N\left(\frac{x}{N+1}\right)^{1/r},
 \qquad
 \ell_i=\lfloor\varphi(i)\rfloor
 \quad(1\le i\le N+1).
 \label{eq:q122-wall}
\end{equation}
We take $N$ large enough that $\ell_1\ge1$; also,
$\ell_{N+1}=N$.  We call $(\ell_i)_{i=1}^N$ the \emph{lower wall} and
$\varphi$ its smooth profile.

The poset $P_{N,r}$ has two chains
\[
 A_1<\cdots<A_N,
 \qquad
 B_1<\cdots<B_N,
\]
and the cross-relations
\begin{equation}
 B_j<A_i\quad\Longleftrightarrow\quad j\le\ell_i.
 \label{eq:q122-relations}
\end{equation}
There are no other cross-relations.
\end{definition}

Since $(\ell_i)$ is nondecreasing, these relations define a poset of
width two.  For the remainder of this section, unless otherwise specified,
write $P=P_{N,r}$.

\subsection{Encoding linear extensions}

\begin{definition}[Height sequence]
\label{def:q122-height-sequence}
Choose a linear extension $f$ of $P$ uniformly at random, and write
$h=h_P$.  For each $i$, let
\[
 K_i=\#\{j:B_j\text{ occurs before }A_i\text{ in }f\}.
\]
The sequence $(K_i)_{i=1}^N$ is the \emph{height sequence} of $f$.
\end{definition}

The relations inside the two chains and the cross-relations
\eqref{eq:q122-relations} imply
\begin{equation}
 \ell_i\le K_i\le N,
 \qquad K_1\le\cdots\le K_N.                            \label{eq:q122-K}
\end{equation}
Thus saying that the height sequence lies above the lower wall means
exactly
\[
 K_i\ge\ell_i\qquad(1\le i\le N).
\]

Conversely, every integer sequence satisfying \eqref{eq:q122-K}
determines exactly one linear extension: insert $A_i$ immediately after
the first $K_i$ elements of the $B$-chain.  Hence a uniformly random
linear extension is equivalent to a uniformly random height sequence
satisfying \eqref{eq:q122-K}.

\begin{definition}
\label{def:q122-coordinates}
For $0\le j\le N$, set
\begin{equation}
 X_j=\#\{i:K_i=j\},\qquad
 M_j=\sum_{t=0}^jX_t,
 \qquad
 c_j=\#\{i:1\le i\le N,\ \ell_i\le j\}.               \label{eq:q122-occupancy}
\end{equation}
Set $c_j=0$ for $j<0$.  Thus $X_j$ counts the $A$-elements at height
$j$, $M_j$ counts those at height at most $j$, and $c_j$ counts the wall
heights at most $j$.  In particular, $X_N$ counts the $A$-elements
occurring after $B_N$.

For $1\le i\le N$, set
\begin{equation}
 G_i=h(A_i)-h(B_{\ell_i}),                              \label{eq:q122-Gdefinition}
\end{equation}
and, whenever $i=1$ or $\ell_{i-1}<\ell_i$, set
\[
 C_i=\{B_1,\ldots,B_{\ell_i},A_i,\ldots,A_N\}.
\]
We call $G_i$ the \emph{switch spacing} and $C_i$ a \emph{mixed chain}.
\end{definition}

The two maximal-chain types and their relevant expected-rank spacings are
summarized in \cref{fig:q122-maximal-chain-spacings}; the following lemma
makes the reduction precise.

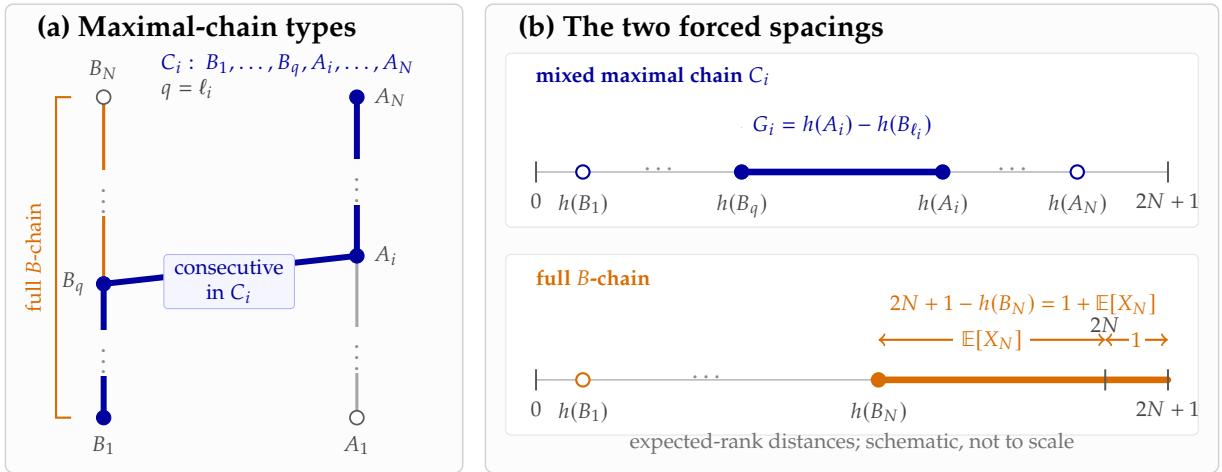
\begin{figure}[htbp]
\centering
\begin{tikzpicture}[
  x=1cm,
  y=1cm,
  font=\small,
  panel/.style={
    draw=black!18,fill=black!1,rounded corners=3pt,line width=.55pt
  },
  chain/.style={draw=black!35,line width=1.15pt,line cap=round},
  fullb/.style={draw=orange!85!black,line width=1.15pt,line cap=round},
  selected/.style={draw=blue!62!black,line width=2.25pt,line cap=round},
  gapline/.style={draw=blue!62!black,line width=2.5pt,line cap=round},
  endgap/.style={draw=orange!85!black,line width=2.5pt,line cap=round},
  posetpoint/.style={
    circle,draw=black!65,fill=white,line width=.65pt,
    inner sep=0pt,minimum size=5.2pt
  },
  selectedpoint/.style={
    circle,draw=blue!62!black,fill=blue!62!black,
    inner sep=0pt,minimum size=5.4pt
  },
  rankpoint/.style={
    circle,draw=blue!62!black,fill=white,line width=.8pt,
    inner sep=0pt,minimum size=5pt
  },
  bendpoint/.style={
    circle,draw=orange!85!black,fill=white,line width=.8pt,
    inner sep=0pt,minimum size=5pt
  },
  endpoint/.style={draw=black!65,line width=.75pt},
  tinylabel/.style={font=\scriptsize,text=black!72},
  card/.style={
    draw=black!10,fill=white,rounded corners=2pt,line width=.45pt
  }
]

% (a) The two types of maximal chains.
\draw[panel] (-.10,-.15) rectangle (5.95,6.05);
\node[anchor=west,font=\bfseries] at (.18,5.72)
  {(a) Maximal-chain types};

\coordinate (B1) at (1.20,.58);
\coordinate (Bq) at (1.20,2.35);
\coordinate (BN) at (1.20,4.82);
\coordinate (A1) at (4.55,.58);
\coordinate (Ai) at (4.55,2.72);
\coordinate (AN) at (4.55,4.82);

% The two ambient chains and the two maximal-chain types.
\draw[fullb] (B1)--(BN);
\draw[chain] (A1)--(AN);
\draw[selected] (B1)--(Bq)--(Ai)--(AN);

\draw[orange!85!black,line width=.8pt]
  (.77,.58)--(.57,.58)--(.57,4.82)--(.77,4.82);
\node[font=\scriptsize,text=orange!85!black,rotate=90]
  at (.27,2.70) {full $B$-chain};

\node[selectedpoint,label={[tinylabel]below:$B_1$}] at (B1) {};
\node[selectedpoint,label={[tinylabel]left:$B_q$}] at (Bq) {};
\node[posetpoint,label={[tinylabel]above:$B_N$}] at (BN) {};
\node[posetpoint,label={[tinylabel]below:$A_1$}] at (A1) {};
\node[selectedpoint,label={[tinylabel]right:$A_i$}] at (Ai) {};
\node[selectedpoint,label={[tinylabel]right:$A_N$}] at (AN) {};

\node[fill=black!1,inner sep=1pt,text=black!45] at (1.20,1.43) {$\vdots$};
\node[fill=black!1,inner sep=1pt,text=black!45] at (1.20,3.55) {$\vdots$};
\node[fill=black!1,inner sep=1pt,text=black!45] at (4.55,1.48) {$\vdots$};
\node[fill=black!1,inner sep=1pt,text=black!45] at (4.55,3.70) {$\vdots$};

\node[
  draw=blue!30,fill=blue!4,rounded corners=1.5pt,
  inner xsep=4pt,inner ysep=2pt,font=\scriptsize,
  align=center,text=blue!62!black
] at (2.86,2.38) {consecutive\\in $C_i$};
\node[font=\scriptsize,text=blue!62!black,anchor=west]
  at (1.82,5.24) {$C_i:\ B_1,\ldots,B_q,A_i,\ldots,A_N$};
\node[tinylabel,anchor=west] at (1.82,4.92)
  {$q=\ell_i$};

% (b) The corresponding expected-rank spacings.
\draw[panel] (6.25,-.15) rectangle (15.95,6.05);
\node[anchor=west,font=\bfseries] at (6.55,5.72)
  {(b) The two forced spacings};

% The switch spacing in C_i.
\draw[card] (6.52,3.10) rectangle (15.68,5.40);
\node[anchor=west,font=\scriptsize\bfseries,text=blue!62!black]
  at (6.78,5.10) {mixed maximal chain $C_i$};
\draw[black!25,line width=.65pt] (6.92,3.83)--(15.28,3.83);
\draw[gapline] (9.64,3.83)--(12.30,3.83);

\draw[endpoint] (6.92,3.69)--(6.92,3.97);
\node[tinylabel,below] at (6.92,3.68) {$0$};
\node[rankpoint] at (7.54,3.83) {};
\node[tinylabel,below] at (7.54,3.68) {$h(B_1)$};
\node[text=black!45] at (8.57,3.83) {$\cdots$};
\node[rankpoint,fill=blue!62!black] at (9.64,3.83) {};
\node[tinylabel,below] at (9.64,3.68) {$h(B_q)$};
\node[rankpoint,fill=blue!62!black] at (12.30,3.83) {};
\node[tinylabel,below] at (12.30,3.68) {$h(A_i)$};
\node[text=black!45] at (13.25,3.83) {$\cdots$};
\node[rankpoint] at (14.08,3.83) {};
\node[tinylabel,below] at (14.08,3.68) {$h(A_N)$};
\draw[endpoint] (15.28,3.69)--(15.28,3.97);
\node[tinylabel,below] at (15.28,3.68) {$2N+1$};

\draw[<->,draw=blue!62!black,line width=.9pt]
  (9.64,4.42)--(12.30,4.42);
\node[
  fill=white,inner xsep=4pt,font=\scriptsize,
  text=blue!62!black
] at (10.97,4.42)
  {$G_i=h(A_i)-h(B_{\ell_i})$};

% The terminal spacing in the full B-chain.
\draw[card] (6.52,.35) rectangle (15.68,2.74);
\node[anchor=west,font=\scriptsize\bfseries,text=orange!85!black]
  at (6.78,2.43) {full $B$-chain};
\draw[black!25,line width=.65pt] (6.92,1.08)--(15.28,1.08);
\draw[endgap] (11.45,1.08)--(15.28,1.08);

\draw[endpoint] (6.92,.94)--(6.92,1.22);
\node[tinylabel,below] at (6.92,.93) {$0$};
\node[bendpoint] at (7.54,1.08) {};
\node[tinylabel,below] at (7.54,.93) {$h(B_1)$};
\node[text=black!45] at (9.20,1.08) {$\cdots$};
\node[bendpoint,fill=orange!85!black] at (11.45,1.08) {};
\node[tinylabel,below] at (11.45,.93) {$h(B_N)$};
\draw[endpoint] (14.45,.94)--(14.45,1.22);
\node[tinylabel,above,fill=white,inner sep=1pt]
  at (14.45,1.67) {$2N$};
\draw[endpoint] (15.28,.94)--(15.28,1.22);
\node[tinylabel,below] at (15.28,.93) {$2N+1$};

\draw[<->,draw=orange!85!black,line width=.8pt]
  (11.45,1.61)--(14.45,1.61);
\node[fill=white,inner xsep=3pt,font=\scriptsize,
  text=orange!85!black] at (12.95,1.61) {$\E[X_N]$};
\draw[<->,draw=orange!85!black,line width=.8pt]
  (14.45,1.61)--(15.28,1.61);
\node[fill=white,inner xsep=1pt,font=\scriptsize,
  text=orange!85!black] at (14.865,1.61) {$1$};
\node[font=\scriptsize,text=orange!85!black]
  at (13.365,2.08) {$2N+1-h(B_N)=1+\E[X_N]$};

\node[font=\scriptsize,text=black!55,anchor=south]
  at (11.10,-.03) {expected-rank distances; schematic, not to scale};
\end{tikzpicture}

\caption{The two types of maximal chains and their relevant
spacings: $G_i=h(A_i)-h(B_{\ell_i})$ for $C_i$, and
$1+\E[X_N]=2N+1-h(B_N)$ for the full $B$-chain.}
\label{fig:q122-maximal-chain-spacings}
\end{figure}

\begin{lemma}
\label{lem:q122-maxchains}
\begin{equation}
 \min_{\varnothing\ne C\text{ a maximal chain}}\gap_P(C)
 \ge
 \min\left\{1+\E[X_N],\
 \min_{\substack{1\le i\le N\\ i=1\text{ or }\ell_{i-1}<\ell_i}}G_i
 \right\}.
 \label{eq:q122-maxchains}
\end{equation}
\end{lemma}

\begin{proof}
As shown in \cref{fig:q122-maximal-chain-spacings}, the maximal chains are
the full $B$-chain and the mixed chains $C_i$ from
\cref{def:q122-coordinates}, where $i=1$ or
$\ell_{i-1}<\ell_i$.
The terminal spacing of the full $B$-chain is
$2N+1-h(B_N)=1+\E[X_N]$, while $B_{\ell_i}$ and $A_i$ are consecutive in
$C_i$ and leave the spacing $G_i$.
\end{proof}

We now state the constraints on the coordinates from
\cref{def:q122-coordinates} and the rank identities used to estimate the
two quantities in \eqref{eq:q122-maxchains}.
Because $(K_i)$ is nondecreasing, it is determined by the numbers
$(X_j)$.  The wall inequalities $K_i\ge\ell_i$ are equivalent to
\begin{equation}
 X_j\in\mathbb Z_{\ge0},\qquad \sum_{j=0}^NX_j=N,\qquad
 M_j\le c_j\quad(0\le j\le N).                         \label{eq:q122-feasible}
\end{equation}

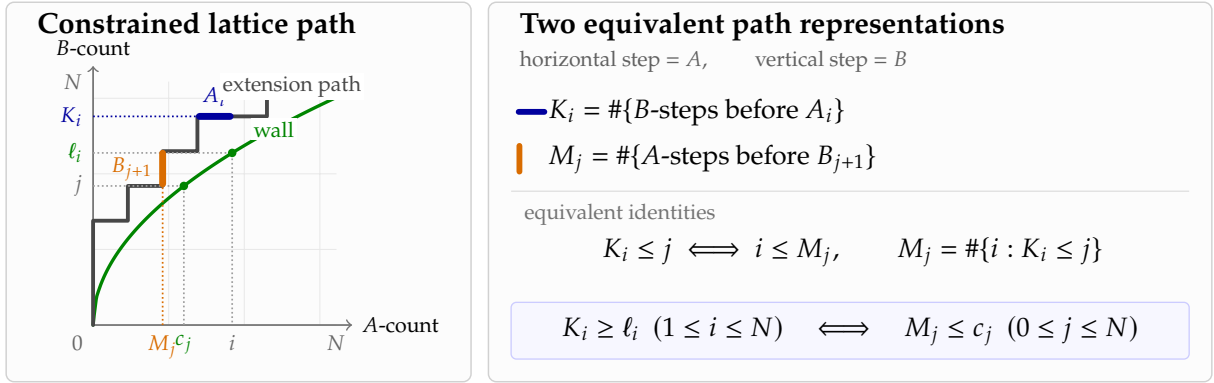
\begin{figure}[!ht]
\centering
\begin{tikzpicture}[
  font=\small,
  panel/.style={
    draw=black!18,fill=black!1,rounded corners=3pt,line width=.55pt
  },
  lattice/.style={draw=black!9,line width=.35pt},
  axis/.style={draw=black!55,line width=.65pt,->},
  path/.style={
    draw=black!72,line width=1.45pt,
    line cap=round,line join=round
  },
  wall/.style={
    draw=green!52!black,line width=1.25pt,
    line cap=round
  },
  kguide/.style={
    draw=blue!62!black,densely dotted,line width=.65pt
  },
  mguide/.style={
    draw=orange!85!black,densely dotted,line width=.65pt
  },
  wallguide/.style={
    draw=black!38,densely dotted,line width=.65pt
  },
  astep/.style={
    draw=blue!62!black,line width=2.6pt,line cap=round
  },
  bstep/.style={
    draw=orange!85!black,line width=2.6pt,line cap=round
  }
]

\draw[panel] (-.10,-.62) rectangle (6.05,4.40);
\node[anchor=west,font=\bfseries] at (.18,4.08)
  {Constrained lattice path};

\begin{scope}[shift={(1.05cm,.13cm)},x=.46cm,y=.46cm]
  \draw[lattice] (0,0) grid (7,7);
  \draw[axis] (0,0)--(7.45,0)
    node[right,font=\scriptsize] {$A$-count};
  \draw[axis] (0,0)--(0,7.45)
    node[above,font=\scriptsize] {$B$-count};

  % The curved wall represents the discrete lower wall directly.
  \draw[wall]
    plot[domain=0:7,samples=80,smooth]
      (\x,{7*sqrt(\x/8)});

  % A feasible extension path visibly above the wall:
  % (K_1,\ldots,K_7)=(3,4,5,6,6,7,7).
  \draw[path]
    (0,0)--(0,3)--(1,3)--(1,4)--(2,4)--(2,5)
    --(3,5)--(3,6)--(5,6)--(5,7)--(7,7);

  % Here i=j=4.  The wall height ell_i is read at x=i, whereas
  % the inverse-wall coordinate c_j is read where the wall has height j.
  % For the displayed curve these are distinct points:
  % (i,ell_i)=(4,7/sqrt 2) and (c_j,j)=(128/49,4).
  \draw[wallguide] (0,4.95)--(4,4.95)--(4,0);
  \draw[wallguide] (0,4)--(2.61,4)--(2.61,0);
  \draw[kguide] (0,6)--(4,6);
  \draw[mguide] (2,0)--(2,4);
  \fill[green!52!black] (4,4.95) circle (1.6pt);
  \fill[green!52!black] (2.61,4) circle (1.6pt);
  \draw[astep] (3.06,6)--(3.94,6);
  \draw[bstep] (2,4.06)--(2,4.94);

  \node[font=\scriptsize,text=blue!62!black,above]
    at (3.50,6) {$A_i$};
  \node[font=\scriptsize,text=orange!85!black,left]
    at (2,4.50) {$B_{j+1}$};
  \node[font=\scriptsize,text=blue!62!black,left]
    at (0,6) {$K_i$};
  \node[font=\scriptsize,text=green!52!black,left]
    at (0,4.95) {$\ell_i$};
  \node[font=\scriptsize,text=black!58,left]
    at (0,4) {$j$};
  \node[font=\scriptsize,text=orange!85!black,below]
    at (2,0) {$M_j$};
  \node[font=\scriptsize,text=green!52!black,below]
    at (2.61,0) {$c_j$};
  \node[font=\scriptsize,text=black!58,below]
    at (4,0) {$i$};
  \node[font=\scriptsize,text=green!52!black,fill=white,
    inner sep=1pt] at (5.20,5.64) {wall};
  \node[font=\scriptsize,text=black!72,fill=white,
    inner sep=1pt] at (5.72,6.87) {extension path};
  \node[font=\scriptsize,text=black!60,below left] at (0,0) {$0$};
  \node[font=\scriptsize,text=black!60,below] at (7,0) {$N$};
  \node[font=\scriptsize,text=black!60,left] at (0,7) {$N$};
\end{scope}

\draw[panel] (6.28,-.62) rectangle (15.85,4.40);
\node[anchor=west,font=\bfseries] at (6.56,4.08)
  {Two equivalent path representations};
\node[anchor=west,font=\scriptsize,text=black!62]
  at (6.56,3.62)
  {horizontal step $=A$, \qquad vertical step $=B$};

\draw[blue!62!black,line width=2.2pt,line cap=round]
  (6.68,2.95)--(7.02,2.95);
\node[anchor=west] at (6.94,2.95)
  {$K_i=\#\{B\text{-steps before }A_i\}$};

\draw[orange!85!black,line width=2.2pt,line cap=round]
  (6.69,2.16)--(6.69,2.50);
\node[anchor=west] at (6.94,2.33)
  {$M_j=\#\{A\text{-steps before }B_{j+1}\}$};

\draw[black!12,line width=.5pt] (6.58,1.90)--(15.55,1.90);
\node[anchor=west,font=\scriptsize,text=black!55]
  at (6.62,1.60) {equivalent identities};
\node at (11.10,1.10)
  {$\displaystyle
    K_i\le j\ \Longleftrightarrow\ i\le M_j,
    \qquad
    M_j=\#\{i:K_i\le j\}$};

\draw[
  draw=blue!20,fill=blue!3,
  rounded corners=2pt,line width=.5pt
] (6.58,-.32) rectangle (15.55,.43);
\node at (11.065,.055)
  {$\displaystyle
    K_i\ge\ell_i\ \ (1\le i\le N)
    \quad\Longleftrightarrow\quad
    M_j\le c_j\ \ (0\le j\le N)$};
\end{tikzpicture}

\caption{The dual coordinate descriptions of one linear extension.
In the left panel, the green curve schematically represents the wall, and the black
lattice path is a feasible extension path lying weakly northwest of it.
The two wall coordinates are read at different points: $\ell_i$ is the
wall height at $i$, whereas $c_j$ is its horizontal coordinate at height
$j$.  The marked steps illustrate $\ell_i\le K_i$ and $M_j\le c_j$.}
\label{fig:q122-dual-coordinates}
\end{figure}

The correspondence between height sequences satisfying \eqref{eq:q122-K}
and vectors satisfying \eqref{eq:q122-feasible} is bijective.  Therefore
the feasible vectors $(X_0,\ldots,X_N)$ are also uniform.

For every linear extension,
\begin{equation}
 f(A_i)=i+K_i,
 \qquad
 f(B_j)=j+M_{j-1}.                                     \label{eq:q122-ranks}
\end{equation}
Indeed, the elements preceding $A_i$ are $A_1,\ldots,A_{i-1}$ and
$K_i$ elements of the $B$-chain.  Likewise, exactly $M_{j-1}$ elements
of the $A$-chain precede $B_j$.  Taking expectations gives
\begin{equation}
 h(A_i)=i+\E[K_i],
 \qquad
 h(B_j)=j+\E[M_{j-1}].                                  \label{eq:q122-means}
\end{equation}

\begin{lemma}
\label{lem:q122-Gidentity}
For $1\le i\le N$ and $q=\ell_i$,
\begin{equation}
 G_i=i+\E[K_i]-q-\E[M_{q-1}]
 =1+\E\left[(K_i-q)+\#\{t<i:K_t\ge q\}\right].
 \label{eq:q122-Gidentity}
\end{equation}
\end{lemma}

\begin{proof}
The first equality follows from \eqref{eq:q122-means} and
$G_i=h(A_i)-h(B_q)$.  Since $K_t\ge\ell_t\ge q$ for every $t\ge i$,
\[
 M_{q-1}=\#\{t<i:K_t<q\}.
\]
Among the indices $t<i$, the remaining
$i-1-M_{q-1}$ satisfy $K_t\ge q$.  Substitution gives the second
equality.  The two terms inside the expectation count, respectively,
the additional $B$-elements and the earlier $A$-elements lying strictly
between $B_q$ and $A_i$.
\end{proof}

\begin{proposition}
\label{prop:q122-clearance}
For every sufficiently large integer $r$, there is $N_0(r)$ such that
for every $N\ge N_0(r)$,
\[
 \min\left\{1+\E[X_N],\
 \min_{1\le i\le N}G_i
 \right\}
 \ge \frac r5.
\]
\end{proposition}

\begin{proof}[Proof of \cref{thm:q122-main} assuming
\cref{prop:q122-clearance}]
Choose an integer $r>5L$ for which \cref{prop:q122-clearance} holds, and
then choose $N\ge N_0(r)$.  By construction, $\width(P)=2$.  Every
nonempty chain $C\subseteq P$ extends to a maximal chain, and adding
elements only subdivides expected-rank spacings.  Hence
\cref{lem:q122-maxchains,prop:q122-clearance} give
\[
 \gap_P(C)\ge r/5>L.
\]
\end{proof}

\subsection{Elementary estimates}

The following lemmas concern the uniform feasible vectors in
\eqref{eq:q122-feasible}.  The first two control the final spacing
$1+\E[X_N]$ of the full $B$-chain; the third controls the two terms in
\eqref{eq:q122-Gidentity}.
Recall from \cref{def:q122-coordinates} that the wall-count function is
\begin{equation*}
 c_s=\#\{t:1\le t\le N,\ \ell_t\le s\}
 \quad(0\le s\le N),
 \qquad c_s=0\quad(s<0).
\end{equation*}

\begin{lemma}[Terminal-spacing estimates]\label{lem:q122-terminal-estimates}
For the uniform model \eqref{eq:q122-feasible},
\begin{align}
 \E[X_0]&\le\E[X_1]\le\cdots\le\E[X_N],                \label{eq:q122-monotone}\\
 (N-q)\E[X_N]&\ge N-c_q\qquad(0\le q<N).              \label{eq:q122-XNbound}
\end{align}
\end{lemma}

\begin{proof}
Interchanging coordinates $X_j$ and $X_{j+1}$ defines an injection from
the feasible vectors with $X_j>X_{j+1}$ into those with
$X_j<X_{j+1}$.  Indeed, only $M_j$ changes, and it decreases; moreover,
the injection preserves $|X_{j+1}-X_j|$.  Therefore the sum of
$X_{j+1}-X_j$ over all feasible vectors is nonnegative, proving
\eqref{eq:q122-monotone}.  Consequently,
\[
 (N-q)\E[X_N]
 \ge\sum_{j=q+1}^N\E[X_j]
 =N-\E[M_q]
 \ge N-c_q,
\]
where the final inequality follows because $M_q\le c_q$ for every
feasible vector.  This proves \eqref{eq:q122-XNbound}.
\end{proof}

\begin{lemma}
\label{lem:q122-endpoint-asymptotics}
For $0\le q\le N$,
\begin{equation}
 c_q=\min\left\{N,
 \left\lceil (N+1)\left(\frac{q+1}{N}\right)^r\right\rceil-1
 \right\}.                                               \label{eq:q122-inverse}
\end{equation}
Moreover, for fixed $r$,
\begin{equation}
 \liminf_{N\to\infty}\frac{\E[X_N]}{r}
 \ge 1-\left(1-\frac1r\right)^r>1-e^{-1}.               \label{eq:q122-XNasymp}
\end{equation}
\end{lemma}

\begin{proof}
Equation \eqref{eq:q122-inverse} follows by inverting
\eqref{eq:q122-wall}.  Taking $q=N-\lfloor N/r\rfloor$ in
\eqref{eq:q122-XNbound} and using \eqref{eq:q122-inverse} gives
\eqref{eq:q122-XNasymp}.
\end{proof}

\begin{lemma}[Switch-spacing estimates]\label{lem:q122-switch-estimates}
For the uniform model \eqref{eq:q122-feasible},
\begin{equation}
 \E[K_i-\ell_i]\ge\frac{\ell_{i+1}-\ell_i}{2}
 \qquad(1\le i\le N).                                  \label{eq:q122-vertical}
\end{equation}
For every value $q$ taken by $(\ell_t)_{t=1}^N$, let $i$ be the first
index for which $\ell_i=q$, and let
$\Delta_q=c_{q-1}-c_{q-2}=\#\{t:\ell_t=q-1\}$.  Then
\begin{equation}
 \E\left[\#\{t<i:K_t\ge q\}\right]
 \ge\frac{\Delta_q}{2}.                                 \label{eq:q122-horizontal}
\end{equation}
\end{lemma}

\begin{proof}
A state $(a,b)$ represents a prefix of the extension word that has used
exactly $a$ letters $A$ and $b$ letters $B$.  Let $F(a,b)$ be the number
of feasible word prefixes from $(0,0)$ to $(a,b)$, and let $T(a,b)$ be
the number of feasible suffixes from $(a,b)$ to $(N,N)$.  The number of
extensions satisfying $K_i=k$ is
$F(i-1,k)T(i,k)$.  On
$\ell_i\le k\le\ell_{i+1}$ the suffix count $T(i,k)$ is constant:
before $B_{\ell_{i+1}}$ appears, the only possible intervening moves are
$B$-moves.  The prefix count $F(i-1,k)$ is nondecreasing, because
appending a $B$-move gives an injection from prefixes ending at height
$k$ to prefixes ending at height $k+1$.  Thus the numbers of extensions
at the offsets
\[
 0,1,\ldots,\ell_{i+1}-\ell_i
\]
are nondecreasing.  Their mean offset is at least half the length of this
interval, and extensions with larger offsets only increase the mean.
This proves \eqref{eq:q122-vertical}.

Finally, consider extensions in which $B_q$ is inserted after exactly
$m$ letters $A$.  Their number is $F(m,q-1)T(m,q)$.  For
$c_{q-2}\le m\le c_{q-1}$ the first factor is constant and the second is
nonincreasing.  For the first assertion, appending $A_{m+1}$ gives a
bijection between prefixes ending at $(m,q-1)$ and at $(m+1,q-1)$:
the inverse deletion is valid because, once more than $c_{q-2}$ letters
$A$ have appeared, $B_{q-1}$ must already have appeared.  The second
assertion follows because prepending the
available letter $A_{m+1}$ injects suffixes from $(m+1,q)$ into suffixes
from $(m,q)$.  Hence the numbers of extensions at the values
\[
 c_{q-1}-m=0,1,\ldots,\Delta_q
\]
are nondecreasing.  The expected value of $c_{q-1}-m$ is therefore at
least $\Delta_q/2$; values $m<c_{q-2}$ only increase this difference.
Since $i=c_{q-1}+1$, this difference is exactly
$\#\{t<i:K_t\ge q\}$, proving \eqref{eq:q122-horizontal}.
\end{proof}

\begin{lemma}
\label{lem:q122-switch-spacing-bound}
If $1\le i\le N$ and either $i=1$ or $\ell_{i-1}<\ell_i$, then
\begin{equation}
 G_i\ge
 1+\frac{\ell_{i+1}-\ell_i}{2}
  +\frac{c_{\ell_i-1}-c_{\ell_i-2}}{2}.
  \label{eq:q122-Gdet}
\end{equation}
\end{lemma}

\begin{proof}
Put $q=\ell_i$.  Since $(\ell_t)$ is nondecreasing, the hypothesis on
$i$ makes $i$ the first index for which $\ell_i=q$.  Substituting
\eqref{eq:q122-vertical} and \eqref{eq:q122-horizontal} into
\eqref{eq:q122-Gidentity} proves \eqref{eq:q122-Gdet}.
\end{proof}

\subsection{Proof of the maximal-chain spacing bounds}

Before providing the rigorous proof of \cref{prop:q122-clearance}, we
briefly explain the roadmap and intuition behind the proof.  To prove
\cref{prop:q122-clearance}, we need lower bounds for two quantities:
\begin{itemize}
 \item $1+\E[X_N]$: the lower bound is proved in
       \cref{lem:q122-endpoint-asymptotics}.
 \item $G_i$ for every mixed maximal chain, which is a more complicated
       case.
\end{itemize}

To prove a lower bound for $G_i$, we split the analysis into three parts.
Recall that feasible linear extensions correspond to paths constrained to
remain above the lower wall.  We divide the relevant indices according to
the slope of the wall near $A_i$:
\begin{itemize}
 \item The wall is steep near $A_i$: $\varphi'(i)\ge8r$.
       (Steep regime)
 \item The wall has intermediate slope near $A_i$:
       $2/r<\varphi'(i)<8r$.  (Middle regime)
 \item The wall is nearly flat near $A_i$: $\varphi'(i)\le2/r$.
       (Flat regime)
\end{itemize}

The steep and flat cases are captured by two quantities:
\begin{itemize}
 \item $\E[K_i-\ell_i]$ is large when the wall is steep.
 \item $(c_{\ell_i-1}-c_{\ell_i-2})/2$ is large when the wall is flat.
\end{itemize}

In the middle case, the wall is close to a straight line on a window of
length $m$ around $A_i$.  After conditioning on the endpoint heights, a
midpoint estimate shows that a path constrained above this line has
expected clearance at least a positive constant, depending on $r$, times
$\sqrt m$.  This tends to infinity with $N$ and eventually gives the
required lower bound for $\E[K_i-\ell_i]$.

\begin{proof}[Proof of \cref{prop:q122-clearance}]
Fix a sufficiently large integer $r$.  \Cref{lem:q122-endpoint-asymptotics}
already shows, once $N$ is sufficiently large
depending on $r$, that
\begin{equation}
 1+\E[X_N]\ge\frac r5.                                 \label{eq:q122-endpoint}
\end{equation}
It remains to prove $G_i\ge r/5$ for every $1\le i\le N$.  Let $j$ be
the first index with $\ell_j=\ell_i$.  Since $A_j<\cdots<A_i$,
\[
 G_i-G_j=h(A_i)-h(A_j)\ge i-j.
\]
Thus it suffices to consider $j=1$ or $\ell_{j-1}<\ell_j$.  We divide
these indices into the following three cases.

\begin{figure}[!ht]
\centering
\begin{tikzpicture}[
  x=1cm,y=1cm,font=\small,
  panel/.style={
    draw=black!28,rounded corners=2pt,line width=.65pt
  },
  divider/.style={
    draw=black!22,densely dashed,line width=.55pt
  },
  axis/.style={draw=black!55,line width=.7pt,->},
  regime/.style={
    rounded corners=2pt,line width=.55pt,
    minimum width=6.55cm,minimum height=1.35cm,
    text width=5.95cm,align=left,
    inner xsep=8pt,inner ysep=3.5pt
  }
]

% The normalized profile occupies a genuine square.
\begin{scope}[shift={(1cm,.75cm)}]
\begin{scope}[x=5.25cm,y=5.25cm]
  \begin{scope}
    \clip[rounded corners=2pt] (0,0) rectangle (1,1);
    \fill[blue!4]   (0,0) rectangle (.08,1);
    \fill[black!1]  (.08,0) rectangle (.54,1);
    \fill[orange!5] (.54,0) rectangle (1,1);

    \draw[black!6,line width=.35pt]
      (0,.25)--(1,.25) (0,.50)--(1,.50) (0,.75)--(1,.75);
    \draw[black!6,line width=.35pt]
      (.25,0)--(.25,1) (.50,0)--(.50,1) (.75,0)--(.75,1);

    \draw[divider] (.08,0)--(.08,1);
    \draw[divider] (.54,0)--(.54,1);

    \draw[blue!62!black,line width=1.65pt,line cap=round]
      (0,0)
      .. controls (.002,.43) and (.03,.61) .. (.08,.68);
    \draw[black!72,line width=1.65pt,line cap=round]
      (.08,.68)
      .. controls (.18,.80) and (.38,.87) .. (.54,.90);
    \draw[orange!85!black,line width=1.65pt,line cap=round]
      (.54,.90)
      .. controls (.72,.94) and (.90,.985) .. (1,1);
  \end{scope}

  \draw[panel] (0,0) rectangle (1,1);
  \draw[axis] (0,0)--(1.045,0);
  \draw[axis] (0,0)--(0,1.045);

  \node[anchor=north east,font=\scriptsize,text=black!65]
    at (-.012,-.012) {$0$};
  \node[anchor=north,font=\scriptsize,text=black!65]
    at (1,-.012) {$1$};
  \node[anchor=east,font=\scriptsize,text=black!65]
    at (-.012,1) {$1$};

  \node[anchor=north,font=\scriptsize,text=black!70]
    at (.5,-.095) {$i/(N+1)$};
  \node[rotate=90,font=\scriptsize,text=black!70]
    at (-.13,.5) {$\varphi(i)/N$};
\end{scope}
\end{scope}

\node[
  regime,anchor=north west,
  draw=blue!28,fill=blue!3
] at (7.05,6.00) {%
  {\color{blue!62!black}\bfseries
   Steep\hfill $\varphi'(i)\ge 8r$}\\[4pt]
  {\scriptsize Large wall increment:
   $\ell_{i+1}-\ell_i$.}%
};

\node[
  regime,anchor=north west,
  draw=black!20,fill=black!1
] at (7.05,4.15) {%
  {\color{black!72}\bfseries
   Middle\hfill $2/r<\varphi'(i)<8r$}\\[4pt]
  {\scriptsize Large expected difference:
   $\E[K_i-\ell_i]$.}%
};

\node[
  regime,anchor=north west,
  draw=orange!35,fill=orange!4
] at (7.05,2.30) {%
  {\color{orange!85!black}\bfseries
   Flat\hfill $\varphi'(i)\le 2/r$}\\[4pt]
  {\scriptsize Long wall plateau:
   $c_{\ell_i-1}-c_{\ell_i-2}$.}%
};

\end{tikzpicture}
\caption{The normalized profile in $[0,1]^2$ (schematic).  Since
$\varphi'$ decreases, the regimes run from left to right; each box names
the quantity used in that case.}
\label{fig:q122-slope-regimes}
\end{figure}
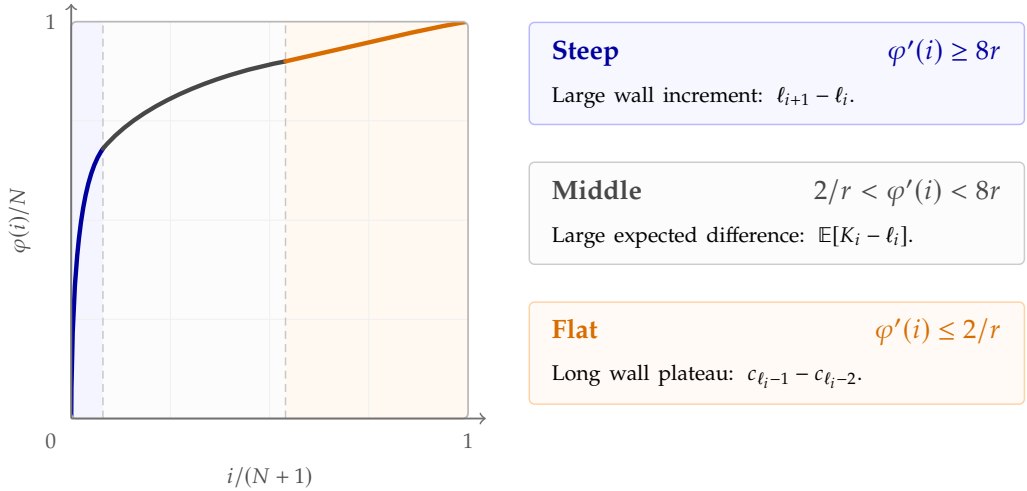

\smallskip
\noindent
\textbf{Steep regime: $\varphi'(i)\ge8r$.}
Since
\[
 \frac{\varphi'(i+1)}{\varphi'(i)}
 =\left(\frac{i}{i+1}\right)^{(r-1)/r}\ge\frac12,
\]
concavity and the inequalities
$\varphi(t)-1<\ell_t\le\varphi(t)$ give
\[
 \ell_{i+1}-\ell_i\ge4r-1.
\]
\Cref{eq:q122-vertical,eq:q122-Gidentity} therefore give
\begin{equation}
 G_i\ge2r-\tfrac12.                                      \label{eq:q122-steep}
\end{equation}

\smallskip
\noindent
\textbf{Flat regime: $\varphi'(i)\le2/r$.}
Put $q=\ell_i$ and
\[
 g(y)=(N+1)(y/N)^r.
\]
The derivative condition implies, uniformly over all flat indices as
$N\to\infty$ with $r$ fixed,
\begin{equation}
 (q/N)^{r-1}\ge\frac12-o(1).                            \label{eq:q122-flat-q}
\end{equation}
The minimum with $N$ in \eqref{eq:q122-inverse} is attained by its second
entry at $q-1$ and $q-2$ once $N$ is sufficiently large.  Indeed,
$q\le N-1$ and
$g(N-1)=N+1-r+O_r(N^{-1})<N$.  Hence the mean-value theorem and
$\lceil x\rceil-\lceil y\rceil\ge x-y-1$ give
\begin{align}
 c_{q-1}-c_{q-2}
 &=\lceil g(q)\rceil-\lceil g(q-1)\rceil\\
 &\ge \frac r2-2.                                       \label{eq:q122-plateau}
\end{align}
By \cref{lem:q122-switch-spacing-bound}, for all sufficiently large $r$ and then
all sufficiently large $N$,
\begin{equation}
 G_i\ge \frac r5.                                       \label{eq:q122-flat}
\end{equation}

\smallskip
\noindent
\textbf{Middle regime: $2/r<\varphi'(i)<8r$.}
\Cref{prop:q122-middle} below gives
\begin{equation}
 G_i\ge r.                                                \label{eq:q122-middle-G}
\end{equation}

Combining the preceding observation with \eqref{eq:q122-steep},
\eqref{eq:q122-flat}, and \eqref{eq:q122-middle-G}, for every
sufficiently large fixed $r$ and then all sufficiently large $N$
(depending on $r$), we obtain
\begin{equation}
 \min_{1\le i\le N}G_i
 \ge\frac r5.                                           \label{eq:q122-minG}
\end{equation}
The bounds \eqref{eq:q122-minG} and \eqref{eq:q122-endpoint} are exactly
the assertion of \cref{prop:q122-clearance}.
\end{proof}

\begin{proposition}[Middle-regime bound]
\label{prop:q122-middle}
Fix $r\ge2$.  For all sufficiently large $N$, depending on $r$, every
integer $i$ with $1\le i\le N$ satisfying
\[
 \frac2r<\varphi'(i)<8r
\]
also satisfies $G_i\ge r$.
\end{proposition}

The proof, together with the two probabilistic lemmas it uses, is given
in Appendix~\ref{app:q122-middle}.

\begin{remark}
The particular power function in \eqref{eq:q122-wall} is not essential.
The same proof scheme applies to any increasing concave family of
profiles that is sufficiently steep near $0$, sufficiently flat near
$1$, and locally approximable by a straight line on growing windows in
the intermediate-slope region, provided that the terminal spacing of
the full $B$-chain also grows.  The steep and flat arguments then change
only by constants, while the middle-regime argument applies on each such
window.
\end{remark}

\section{The gap-entropy counterexample}
\label{sec:entropy-counterexample}

\textcite[Conjecture~2.5]{AiresKahn} conjectured that
$\gap(P_j)\to\infty$ implies
$\tau(P_j)\to\infty$ for every sequence of finite
posets.  In this section, we answer \cref{question:gap-entropy}
negatively by constructing a sequence for which $\gap(P_j)\to\infty$
while $\tau(P_j)<3$.

\subsection{Counterexample construction}

We construct our counterexamples using two operations.
\begin{itemize}
    \item Ordinal sum $Q \oplus R$: Every element of $Q$ is less than every element of $R$, and the internal relations within $Q$ and $R$ are retained.
    \item Parallel sum $Q \sqcup R$: Relations inside $Q$ and $R$ are retained, and no further relations are imposed.
\end{itemize}
%For disjoint posets $Q$ and $R$, their \emph{ordinal sum} $Q\oplus R$
%retains their original orders and declares every element of $Q$ smaller
%than every element of $R$.  Their \emph{parallel sum} $Q\sqcup R$ retains
%their original orders and adds no relation between the two posets.
%Repeated summands always denote distinct isomorphic copies. 
Let $C_k$
denote the chain with $k$ elements. 
Begin with
\[
 P_0=C_2.
\]
For each $r\geq1$, choose a positive integer
$L_r$ and define
\begin{equation}\label{eq:construction}
 R_r=C_{L_r}\oplus P_{r-1}\oplus C_{L_r},
 \qquad
 P_r=R_r\sqcup R_r.
\end{equation}
Here and below, in an ordinal sum $C_L\oplus Q\oplus C_L$, we call the
two copies of $C_L$ the lower and upper padding chains.
Put
\[
 q_r=|P_{r-1}|,
 \qquad
 m_r=|R_r|=2L_r+q_r.
\]

The second and third stages are shown schematically in
\cref{fig:recursive-construction}.

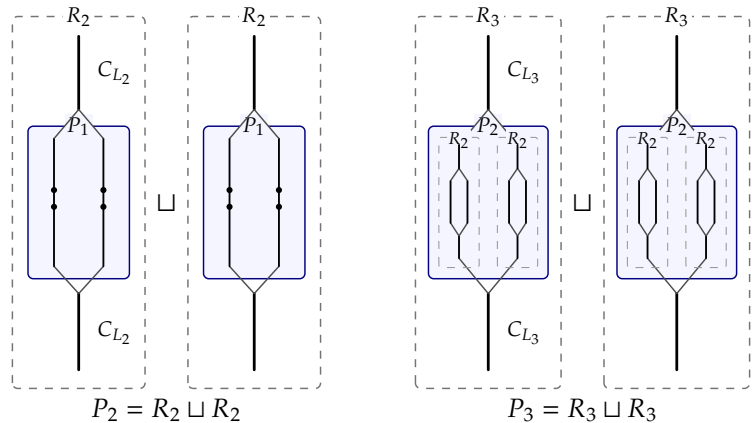
\begin{figure}[htbp]
\centering
\begin{tikzpicture}[
  x=.86cm,
  y=.82cm,
  padding/.style={draw=black,line width=1.05pt,line cap=round},
  innerchain/.style={draw=black,line width=.7pt,line cap=round},
  connector/.style={draw=black!70,line width=.55pt},
  copybox/.style={
    draw=black!55,dashed,rounded corners=2pt,line width=.55pt
  },
  module/.style={
    draw=blue!50!black,fill=blue!4,rounded corners=2pt,
    line width=.55pt
  },
  innerbox/.style={
    draw=black!40,dashed,rounded corners=1pt,line width=.4pt
  }
]

% P_2=R_2\sqcup R_2.
\begin{scope}
  \foreach \dx in {-1.35,1.35} {
    \begin{scope}[shift={(\dx,0)}]
      \draw[copybox] (-1.02,.15) rectangle (1.02,6.15);
      \node[fill=white,inner sep=1pt,font=\scriptsize]
        at (0,6.15) {$R_2$};

      % The lower and upper copies of C_{L_2}.
      \draw[padding] (0,.45) -- (0,1.68);
      \draw[padding] (0,4.65) -- (0,5.83);

      % The middle copy P_1=R_1\sqcup R_1.
      \draw[module] (-.78,1.92) rectangle (.78,4.38);
      \node[fill=blue!4,inner sep=1pt,font=\scriptsize]
        at (0,4.38) {$P_1$};

      \draw[connector] (0,1.68) -- (-.38,2.12);
      \draw[connector] (0,1.68) -- ( .38,2.12);
      \draw[innerchain] (-.38,2.12) -- (-.38,4.18);
      \draw[innerchain] ( .38,2.12) -- ( .38,4.18);
      \draw[connector] (-.38,4.18) -- (0,4.65);
      \draw[connector] ( .38,4.18) -- (0,4.65);

      % The two marked points in each R_1 form P_0=C_2.
      \fill (-.38,3.08) circle (1.15pt);
      \fill (-.38,3.35) circle (1.15pt);
      \fill ( .38,3.08) circle (1.15pt);
      \fill ( .38,3.35) circle (1.15pt);
    \end{scope}
  }

  \node at (0,3.15) {$\sqcup$};
  \node[font=\scriptsize,anchor=west] at (-1.22,1.05)
    {$C_{L_2}$};
  \node[font=\scriptsize,anchor=west] at (-1.22,5.25)
    {$C_{L_2}$};
  \node[font=\small] at (0,-.20)
    {$P_2=R_2\sqcup R_2$};
\end{scope}

% P_3=R_3\sqcup R_3.
\begin{scope}[shift={(6.4,0)}]
  \foreach \dx in {-1.45,1.45} {
    \begin{scope}[shift={(\dx,0)}]
      \draw[copybox] (-1.12,.15) rectangle (1.12,6.15);
      \node[fill=white,inner sep=1pt,font=\scriptsize]
        at (0,6.15) {$R_3$};

      % The lower and upper copies of C_{L_3}.
      \draw[padding] (0,.45) -- (0,1.68);
      \draw[padding] (0,4.65) -- (0,5.83);

      % The middle copy P_2=R_2\sqcup R_2.
      \draw[module] (-.92,1.92) rectangle (.92,4.38);
      \node[fill=blue!4,inner sep=1pt,font=\scriptsize]
        at (0,4.38) {$P_2$};

      \draw[connector] (0,1.68) -- (-.45,2.25);
      \draw[connector] (0,1.68) -- ( .45,2.25);
      \draw[connector] (-.45,4.08) -- (0,4.65);
      \draw[connector] ( .45,4.08) -- (0,4.65);

      % The two copies of R_2 inside P_2.
      \foreach \ux in {-.45,.45} {
        \begin{scope}[shift={(\ux,0)}]
          \draw[innerbox] (-.31,2.10) rectangle (.31,4.20);
          \node[fill=blue!4,inner sep=.5pt,font=\tiny]
            at (0,4.16) {$R_2$};

          \draw[innerchain] (0,2.25) -- (0,2.61);
          \draw[connector] (0,2.61) -- (-.13,2.82);
          \draw[connector] (0,2.61) -- ( .13,2.82);
          \draw[innerchain] (-.13,2.82) -- (-.13,3.50);
          \draw[innerchain] ( .13,2.82) -- ( .13,3.50);
          \draw[connector] (-.13,3.50) -- (0,3.70);
          \draw[connector] ( .13,3.50) -- (0,3.70);
          \draw[innerchain] (0,3.70) -- (0,4.08);
        \end{scope}
      }
    \end{scope}
  }

  \node at (0,3.15) {$\sqcup$};
  \node[font=\scriptsize,anchor=west] at (-1.32,1.05)
    {$C_{L_3}$};
  \node[font=\scriptsize,anchor=west] at (-1.32,5.25)
    {$C_{L_3}$};
  \node[font=\small] at (0,-.20)
    {$P_3=R_3\sqcup R_3$};
\end{scope}
\end{tikzpicture}

\caption{The recursive construction of $P_2$ and $P_3$, with order
increasing upward. Lengths are not to scale.}
\label{fig:recursive-construction}
\end{figure}

We answer \cref{question:gap-entropy} negatively by choosing an infinite
sequence of integers $(L_r)_{r\geq1}$ such that the gaps of the resulting
posets $P_r$ tend to infinity while $\tau(P_r)$ remains bounded by a constant.

\begin{theorem}\label{thm:c25-main}
There exists a sequence of positive integers $(L_r)_{r\geq1}$ such that,
for every $r\geq0$,
\begin{equation}\label{eq:mainbounds}
 \tau(P_r)<3,
 \qquad
 \gap(P_r)=\prod_{s=1}^r\frac{2m_s+1}{m_s+1}
            \geq(3/2)^r.
\end{equation}
\end{theorem}

The exponential growth of the poset gap follows easily from
\cref{prop:gap}.  The more challenging part is proving that the hereditary
entropy is bounded by three.  To prove this, we show that the entropy of an
arbitrary set $X$ in the two copies $R_r\sqcup R_r$ behaves additively and
that we can almost recurse using this property; see \cref{lem:onestep}.
Recursing all the way to the initial instances, we obtain the upper bound
$3|X|$ for the entropy of the linear extension induced on $X$; see
\cref{prop:tau}.

\subsection{The exact gap recurrence}

\begin{lemma}[Duplicate parallel sum]
\label{lem:duplicate-parallel-gap}
Let $S_0,S_1$ be disjoint copies of a nonempty $m$-element poset
$S$, and put $P=S_0\sqcup S_1$.  For $y\in S$ and
$i\in\{0,1\}$, let $y_i$ denote the copy of $y$ in $S_i$.  Then
\begin{equation}\label{eq:normalizedmean}
 \frac{h_P(y_i)}{2m+1}=\frac{h_S(y)}{m+1}
 \qquad(i\in\{0,1\}).
\end{equation}
Consequently,
\begin{equation}\label{eq:parallelgap}
 \gap(S_0\sqcup S_1)=\frac{2m+1}{m+1}\gap(S).
\end{equation}
\end{lemma}

\begin{proof}
Conditional on $y_i$ having rank $k$ in $S_i$, its rank in $P$
is the $k$th order statistic of a uniform $m$-element subset of
$\{1,\ldots,2m\}$, whose mean is
$k(2m+1)/(m+1)$~\cite[Theorem~2]{ONeillOrderStatistics}.  Averaging gives
\eqref{eq:normalizedmean}.  Together with the endpoints, this scales
every spacing by $(2m+1)/(m+1)$; the second copy only repeats values.
This proves \eqref{eq:parallelgap}.
\end{proof}

\begin{proposition}\label{prop:gap}
For every $r\geq1$ and every choice of the positive integers
$L_1,\ldots,L_r$,
\begin{equation}\label{eq:gaprecurrence}
 \gap(P_r)=\frac{2m_r+1}{m_r+1}\,\gap(P_{r-1}).
\end{equation}
In particular, $\gap(P_r)\geq(3/2)^r$.
\end{proposition}

\begin{proof}
For every nonempty finite poset $Q$ and $L\geq1$, padding translates its
expected ranks by $L$ and adds only unit spacings.  Since $\gap(Q)\geq1$,
this gives
\begin{equation}\label{eq:ordinalgap}
 \gap(C_L\oplus Q\oplus C_L)=\gap(Q).
\end{equation}
By \eqref{eq:construction}, \eqref{eq:ordinalgap}, and
\cref{lem:duplicate-parallel-gap},
\[
 \gap(P_r)=\frac{2m_r+1}{m_r+1}\gap(R_r)
 =\frac{2m_r+1}{m_r+1}\gap(P_{r-1}).
\]
Finally, $\gap(P_0)=1$ and each multiplier is at least $3/2$; iteration
proves the product in \eqref{eq:mainbounds} and the stated lower bound.
\end{proof}

\subsection{Separation in parallel merge}

A uniform extension of two
parallel $M$-element components is obtained by choosing their internal
extensions independently and merging them according to an independent
uniform word containing $M$ zeros and $M$ ones.  

The symbol $i$ means that the next element is taken from component $i$.
For $M=2L+q$, $i\in\{0,1\}$, and $1\leq k\leq M$, let $\pi_i(k)$ be
the position in the word of the $k$th occurrence of symbol $i$, and define
the integer interval
\begin{equation}
 \mathcal B_i=[\pi_i(L+1),\pi_i(L+q)].             \label{eq:blocks}
\end{equation}
Thus $\mathcal B_i$ is the span of the $q$ non-padding occurrences from
component $i$.

\begin{lemma}[Separation of the non-padded elements]\label{lem:separation}
Fix $q\geq1$.  There are constants $\eta_q>0$ and $L_0(q)$ such that the
following holds for $L\ge L_0(q)$. 
Mark at most $2q$ pairs $(i,k)$ with $i\in\{0,1\}$,
$1\leq k\leq M$, and $k\notin\{L+1,\ldots,L+q\}$;
then, uniformly over the choice of marks, a uniform merge satisfies
\begin{equation}
 \PP\bigl(\mathcal B_0\cap\mathcal B_1\ne\varnothing
   \ \text{or some marked $i$-occurrence lies in $\mathcal B_{1-i}$}\bigr)
 \le \eta_q\,M^{-1/2}.                                 \label{eq:separation-bound}
\end{equation}
\end{lemma}

\begin{proof}
Assign independent uniform $[0,1]$ keys to the occurrences of each
symbol and merge the two sorted lists.  This produces a uniform word
containing $M$ zeros and $M$ ones.  Conditional on the key $x$ of a fixed
occurrence, the number of opposite symbols preceding it has law
$\operatorname{Bin}(M,x)$.  Since $M=2L+q$, the integers from $L$ to
$L+q$ lie in $[M/3,2M/3]$ for all sufficiently large $L$.  The mass at
$j$ is maximized over $x$ at $x=j/M$.  The uniform Bernoulli-sum bound
of \textcite[Theorem~2.1, p.~353]{BaillonCominettiVaisman}
therefore gives an absolute constant $C$ such that
\[
 \sup_{0\le x\le1}\ \max_{L\le j\le L+q}
 \PP(\operatorname{Bin}(M,x)=j)\le C M^{-1/2}.
\]
Either event in \eqref{eq:separation-bound} therefore forces one of the
marked occurrences, or the later of $\pi_0(L+1)$ and $\pi_1(L+1)$, to
have an opposite-symbol count in $[L,L+q]$.  A union bound over at most
$2q+2$ occurrences and $q+1$ counts proves the result with
$\eta_q=C(2q+2)(q+1)$.
\end{proof}

\subsection{The one-step entropy recursion}

For a finite poset $Q$, set
\begin{equation}\label{eq:Tdef}
 \mathcal T_L(Q)=
 (C_L\oplus Q\oplus C_L)\sqcup(C_L\oplus Q\oplus C_L).
\end{equation}
Label the two copies of $C_L\oplus Q\oplus C_L$ by $i=0,1$, and call
them the two components.  Let $Q_i$ be the copy of $Q$ in component $i$;
identify each $Q_i$ with $Q$ when comparing their distributions.  The
two copies of $C_L$ in that component are its lower and upper padding
chains.  Given $X\subseteq\mathcal T_L(Q)$, set $A_i=X\cap Q_i$, and let
$Z_i$ be the subset of $X$ lying in the two padding chains of component
$i$.
Put
\begin{equation}\label{eq:atc}
 a=|A_0|+|A_1|,\qquad
 t=|Z_0|+|Z_1|,\qquad
 c=
 \begin{cases}
  1,&A_0\ne\varnothing\text{ and }A_1\ne\varnothing,\\
  0,&\text{otherwise}.
 \end{cases}
\end{equation}
When $A_i$ is empty, interpret
$\boldsymbol\sigma_{Q_i}|_{A_i}$ as the deterministic empty order, whose
entropy is zero.

\begin{lemma}[One-step entropy recursion]\label{lem:onestep}
For every finite $Q$ and every $\eps>0$, all sufficiently large $L$ satisfy,
simultaneously for every nonempty $X\subseteq\mathcal T_L(Q)$,
\begin{equation}\label{eq:onestep}
 H(\boldsymbol\sigma_{\mathcal T_L(Q)}|_X)
 \leq H(\boldsymbol\sigma_{Q_0}|_{A_0})
      +H(\boldsymbol\sigma_{Q_1}|_{A_1})+2t+c+\eps.
\end{equation}
\end{lemma}

\begin{proof}
Put $q=|Q|$ and
\[
 Y=\boldsymbol\sigma_{\mathcal T_L(Q)}|_X,
 \qquad
 O_i=Y|_{A_i},
 \qquad O=(O_0,O_1).
\]
The internal extensions supplied by the two component extensions are
independent.  Thus $O_0$ and $O_1$ are independent, and $O_i$ has the same
law as $\boldsymbol\sigma_{Q_i}|_{A_i}$, so
$H(O_i)=H(\boldsymbol\sigma_{Q_i}|_{A_i})$.  Moreover,
$O$ is a function of $Y$.  The chain rule, followed by the independence
of $O_0$ and $O_1$, gives
\begin{equation}
 \begin{aligned}
 H(Y)&=H(O)+H(Y\mid O)\\
     &=H(\boldsymbol\sigma_{Q_0}|_{A_0})
       +H(\boldsymbol\sigma_{Q_1}|_{A_1})+H(Y\mid O).
 \end{aligned}                                           \label{eq:entropy-reduction}
\end{equation}
It remains to prove, uniformly over $X$, that
\begin{equation}
 H(Y\mid O)\le2t+c+\eps.                               \label{eq:onestep-conditional}
\end{equation}

Conditional on $O$, the selected elements in component $i$ have the
fixed order
\[
 \text{selected lower padding},\quad O_i,\quad
 \text{selected upper padding}.
\]
The order $Y$ is a binary merge of these two fixed ordered sequences. 
If
$t\ge a$, there are at most $2^{a+t}$ binary merges, so
\[
 H(Y\mid O)\le a+t\le2t.
\]
This proves \eqref{eq:onestep-conditional} in the first case.

Suppose now that $t<a$, and mark the $t$ occurrences corresponding to
$Z_0\cup Z_1$.  Here $\mathcal B_i$ is the span of the occurrences from
$Q_i$, so it is the block corresponding to the selected set $A_i$.
Let $I_X=1$ when $\mathcal B_0$ and $\mathcal B_1$ are disjoint and no
marked $i$-occurrence lies in $\mathcal B_{1-i}$, and let $I_X=0$
otherwise.

On $I_X=1$, each nonempty $A_i$ is a block of $Y$.  Contract these blocks
to labeled markers.  Conditional on $O$, the contracted order is a
fixed-count binary merge of $t+c+1\le t+2$ terms, so it has at most
$2^{t+c}$ possible orders and, together with $O$, determines $Y$.  Hence
\[
 H(Y\mid O,I_X=1)\le t+c.
\]
On $I_X=0$, the uncontracted merge gives
$H(Y\mid O,I_X=0)\le a+t$.

Put $p_X=\PP(I_X=0)$ and
$p_L(q)=\eta_q(2L+q)^{-1/2}$.  Since $t<a\le2q$,
\cref{lem:separation} gives $p_X\le p_L(q)\to0$.  The merge word, and
therefore $I_X$, is independent of $O$.  If we write
$h_2(p)=-p\log_2p-(1-p)\log_2(1-p)$ for binary entropy and take $L$
large enough that $p_L(q)\le1/2$, the chain rule gives
\[
 H(Y\mid O)
 \le h_2(p_X)+(1-p_X)(t+c)+p_X(a+t)
 \le 2t+c+h_2(p_L(q))+2q\,p_L(q).
\]
For all sufficiently large $L$, the last two terms are at most $\eps$.
This proves \eqref{eq:onestep-conditional}; together with
\eqref{eq:entropy-reduction}, it proves \eqref{eq:onestep}.
\end{proof}

\subsection{Uniform hereditary entropy}

Lemma~\ref{lem:onestep} bounds the entropy at stage $r$ by the entropies
of the two copies from stage $r-1$, together with the padding contribution
and an error $\eps_r$.  These errors accumulate when the recurrence is
iterated, so we choose the summable geometric sequence
\begin{equation}\label{eq:eps}
 \eps_r=2^{-r},\qquad \sum_{r\geq1}\eps_r=1.
\end{equation}
Its total is $1$, which keeps the final
entropy coefficient below $3$.
At stage $r\geq1$, after $P_{r-1}$ is fixed, choose $L_r$ large enough for
Lemma~\ref{lem:onestep} with $Q=P_{r-1}$ and $\eps=\eps_r$.

\begin{proposition}\label{prop:tau}
For every $r\geq0$ and every nonempty $X\subseteq P_r$,
\begin{equation}\label{eq:H3}
 H(\boldsymbol\sigma_{P_r}|_X)
 \leq(3-2^{-r})|X|-1<3|X|.
\end{equation}
Thus $\tau(P_r)<3$.
\end{proposition}

\begin{proof}
Put, for $r\geq0$,
\[
 \delta_r=\sum_{s=1}^r\eps_s=1-2^{-r},
\]
where the sum is empty when $r=0$ and hence has value~$0$.
It is enough to prove the following statement by induction on $r$: for
every nonempty $X\subseteq P_r$,
\begin{equation}
 H(\boldsymbol\sigma_{P_r}|_X)\leq(2+\delta_r)|X|-1.
 \label{eq:entropy-induction}
\end{equation}
Since $2+\delta_r=3-2^{-r}$, this is exactly \eqref{eq:H3}.  For $r=0$, every
restriction of $P_0=C_2$ has deterministic order, so
\eqref{eq:entropy-induction} holds.

For the induction step, view
$P_r=\mathcal T_{L_r}(P_{r-1})$ and use the notation
\eqref{eq:atc}.  If $a=0$, then $t=|X|\ge1$, and the restricted order is
a binary merge of two fixed padding sequences.  There are at most $2^t$
such merges, so
\[
 H(\boldsymbol\sigma_{P_r}|_X)\le t\le(2+\delta_r)t-1.
\]

Suppose $a>0$, and let $k\in\{1,2\}$ count the nonempty $A_i$.  The
induction hypothesis applied separately to those $k$ sets contributes
$(2+\delta_{r-1})a-k$.  Together with \cref{lem:onestep}, this gives
\begin{align*}
 H(\boldsymbol\sigma_{P_r}|_X)
 &\leq(2+\delta_{r-1})a-k+2t+c+\eps_r\\
 &\leq(2+\delta_r)a+2t-k+c\\
 &=(2+\delta_r)a+2t-1\\
 &\leq(2+\delta_r)(a+t)-1.
\end{align*}
Here the second line uses
$\delta_r=\delta_{r-1}+\eps_r$ and $\eps_r\le\eps_r a$, the third uses
$c=k-1$, and the last uses $\delta_r\ge0$.  Since $a+t=|X|$, the induction is
complete.
\end{proof}

\section*{Acknowledgements}

The key ideas and theorem were found by ChatGPT 5.6 Sol.  The manuscript
is written by the author with assistance from Codex.
We thank Jeff Kahn and Max Aires for pointing out the application in
\cref{cor:isolated-balance}.
The author was supported by NSF MAI-2501597.

\printbibliography

\clearpage
\appendix
\section{The middle-regime bound}
\label{app:q122-middle}

This appendix proves \cref{prop:q122-middle}.  We begin with the two
probabilistic estimates used in its proof.

\subsection{Midpoint height under a nonnegativity condition}

\begin{lemma}[Midpoint height under a nonnegativity condition]
\label{lem:q122-excursion}
Fix $0<a<b<\infty$ and an integer $\kappa\ge0$.  There are constants
$\gamma>0$ and $m_0$ with the following property.  Let $m\ge m_0$ be
even, let $V$ be a nonnegative integer satisfying $V/m\in[a,b]$, and
choose $(D_1,\ldots,D_m)$ uniformly from
\[
 \left\{(d_1,\ldots,d_m)\in\mathbb Z_{\ge0}^m:
 d_1+\cdots+d_m=V\right\}.
\]
Set
\begin{equation}
 S_t:=\kappa+\sum_{j=1}^t\left(D_j-\frac Vm\right).
 \label{eq:q122-walkcondition}
\end{equation}
Then, under the displayed uniform law,
\begin{equation}
 \E\left[S_{m/2}\,\middle|\,
 S_t\ge0\text{ for every }1\le t\le m\right]
 \ge\gamma\sqrt m.                                     \label{eq:q122-excursionmean}
\end{equation}
The constants depend only on $a$, $b$, and $\kappa$.
\end{lemma}

\begin{proof}
Set $s=V/m$.  Let $\PP_s$ be the joint law of independent random
variables $Y_1^{(s)},\ldots,Y_m^{(s)}$ with
\[
 \PP_s(Y_j^{(s)}=d)
 =\frac1{1+s}\left(\frac{s}{1+s}\right)^d,
 \qquad d\in\mathbb Z_{\ge0}.
\]
Write $\E_s$ for expectation under $\PP_s$.  Each variable has mean $s$,
and, conditional on $\sum_jY_j^{(s)}=V$, their joint law is uniform on
the set of integer tuples displayed in the lemma.

Define
\[
 \widetilde S_t=\kappa+\sum_{j=1}^t(Y_j^{(s)}-s),
 \qquad 0\le t\le m.
\]
Because $ms=V$, the law in the statement is the law of
$(\widetilde S_t)_{t=0}^m$ conditioned on
\[
 \mathcal A_m=
 \{\widetilde S_t\ge0\ (1\le t\le m),\
 \widetilde S_m=\kappa\}.
\]
This event is nonempty: the choice
$d_j=\lceil js\rceil-\lceil(j-1)s\rceil$ has total $V$ and keeps every
partial sum $\sum_{j=1}^t(d_j-s)=\lceil ts\rceil-ts$ nonnegative.
As $s$ ranges over $[a,b]$, the shifted geometric laws
$Y_j^{(s)}-1$ have exponential moments uniformly bounded near the origin,
and their probabilities at $-1,0,1$ are uniformly bounded below by a
positive constant.  These are the uniform hypotheses required in
\cite[Section~2.2]{OttVelenik}.  Moreover,
\[
 (Y_j^{(s)}-1)-\E_s[Y_j^{(s)}-1]=Y_j^{(s)}-s,
\]
so their centered walk is exactly $(\widetilde S_t)$.  After we increase
$m_0$ if necessary, \cite[Lemma~7.4]{OttVelenik}, applied at time $m/2$
with fixed endpoints $\kappa,\kappa$, gives constants $\theta,c_1>0$ such
that
\[
 \PP_s\bigl(\mathcal A_m,\
 \widetilde S_{m/2}\ge\theta\sqrt m\bigr)
 \ge c_1\,m^{-3/2}.
\]
The upper bound \cite[Lemma~6.2]{OttVelenik} gives a constant $C_0>0$
such that, uniformly over the same family,
\[
 \PP_s(\mathcal A_m)\le C_0\,m^{-3/2}.
\]
Dividing the two bounds and using nonnegativity on $\mathcal A_m$ gives
\[
 \E_s[\widetilde S_{m/2}\mid\mathcal A_m]
 \ge\theta\frac {c_1}{C_0}\sqrt m.
\]
Taking $\gamma=\theta c_1/C_0$ proves
\eqref{eq:q122-excursionmean}.
\end{proof}

\subsection{Monotonicity of constrained sequences}

\begin{definition}[Constrained sequence space]
\label{def:q122-constrained-sequences}
For integers $p<q$, endpoint heights $u\le v$, and real lower bounds
$w=(w_p,\ldots,w_q)$, let
\[
 \Omega(u,v;w)=
 \bigl\{(k_p,\ldots,k_q)\in\mathbb Z^{q-p+1}:
 k_p=u,\ k_q=v,\ k_p\le\cdots\le k_q,\ k_t\ge w_t\bigr\}.
\]
Whenever this set is nonempty, give it the uniform probability law.
\end{definition}

The next lemma describes how this uniform law changes when an endpoint or
a lower bound is raised.

\begin{lemma}[Endpoint and lower-bound monotonicity]
\label{lem:q122-monotonicity}
For every function of the interior coordinates that is nondecreasing in
each coordinate, its expectation cannot decrease when the left endpoint
is raised, when the right endpoint is raised, or when any of the lower
bounds $w_t$ is raised, provided the two sequence sets being compared are
nonempty.
\end{lemma}

\begin{proof}
If $q=p+1$, there are no interior coordinates and the claim is immediate.
Assume $q-p\ge2$, discard the two fixed endpoint coordinates, and regard
each law as a uniform law on its interior vector
$(k_{p+1},\ldots,k_{q-1})$.  These interior vectors form a finite
distributive lattice under coordinatewise minimum and maximum.  Its uniform
law is therefore positively associated by the FKG inequality
\cite[Proposition~1]{FortuinKasteleynGinibre}.

If $u\le u'$, then, after this identification of the interior coordinates,
the vectors allowed with left endpoint $u'$ are exactly those allowed with
left endpoint $u$ that satisfy $k_{p+1}\ge u'$.  Similarly, raising one or
more lower bounds restricts the original interior vectors by an increasing
event.  Positive association shows that conditioning on either event cannot
decrease the expectation of an increasing function.

Finally, if $v\le v'$, the vectors allowed with right endpoint $v$ are
exactly those allowed with right endpoint $v'$ that satisfy
$k_{q-1}\le v$.  This event is decreasing.  Its complement is increasing,
so FKG implies that conditioning on it cannot increase the expectation of
an increasing function.  Therefore raising the right endpoint from $v$ to
$v'$ cannot decrease that expectation.
\end{proof}

\subsection{Proof of the proposition}

\begin{proof}[Proof of \cref{prop:q122-middle}]
We first note that
\begin{equation}
 \varphi'(x)=\frac1r\frac{\varphi(x)}x,
 \qquad
 |\varphi''(x)|=\left(1-\frac1r\right)\frac{\varphi'(x)}x.
 \label{eq:q122-derivatives}
\end{equation}
Fix an integer $i$ with $1\le i\le N$ and
$2/r<\varphi'(i)<8r$.  By
\eqref{eq:q122-Gidentity},
\[
 G_i\ge 1+\E[K_i-\ell_i],
\]
because the other term in that identity is nonnegative.  We prove the
uniform estimate
\begin{equation}
 \E[K_i-\ell_i]\ge \gamma_r\sqrt{2R}-4,
 \qquad R=\lfloor N^{1/3}\rfloor,                      \label{eq:q122-middle-clearance}
\end{equation}
where $\gamma_r>0$ depends only on $r$.

The middle-regime inequalities place $i$ away from both endpoints.  More
precisely, for fixed $r$ and all sufficiently large $N$,
\begin{equation}
 a_rN\le i\le b_rN,
 \qquad 0<a_r<b_r<1,                                    \label{eq:q122-middleinterval}
\end{equation}
for constants $a_r,b_r$ depending only on $r$.  Set
\[
 i_-=i-R,\qquad i_+=i+R,\qquad m=i_+-i_-=2R,
 \qquad V=\ell_{i_+}-\ell_{i_-}.
\]
In particular, $1\le i_-<i<i_+\le N$ for all sufficiently large $N$.
Define the affine function joining the two wall points by
\[
 \bar\ell_t=\ell_{i_-}+\frac{t-i_-}{m}
 (\ell_{i_+}-\ell_{i_-}),
 \qquad i_-\le t\le i_+.
\]
Equation~\eqref{eq:q122-derivatives} and
\eqref{eq:q122-middleinterval} give
\begin{equation}
 \sup_{i_-\le t\le i_+}|\varphi''(t)|=O_r(N^{-1}).     \label{eq:q122-curvature}
\end{equation}
For a twice differentiable function on an interval of length $m$, the
deviation from its affine interpolation is at most
$m^2\sup|\varphi''|/8$.  Hence that deviation for $\varphi$ on
$[i_-,i_+]$ is $O_r(R^2/N)=o(1)$.  Allowing for the floor in
$\ell_t=\lfloor\varphi(t)\rfloor$, we obtain, for every integer $t$ in
this interval,
\begin{equation}
 \bar\ell_t-2\le\ell_t\le \bar\ell_t+2
 \qquad(i_-\le t\le i_+,\ t\in\mathbb Z).              \label{eq:q122-chord}
\end{equation}
The derivative varies by $o(1)$ on this interval as $N\to\infty$ with
$r$ fixed.  The middle-regime bounds and the inequalities
$|\ell_t-\varphi(t)|<1$ at $t=i_-,i_+$ therefore imply, uniformly over
all middle indices,
\begin{equation}
 \frac Vm\in\left[\frac1{2r},10r\right]\subset(0,\infty).
 \label{eq:q122-slope}
\end{equation}

Condition on $K_{i_-}=u$ and $K_{i_+}=v$.  Given these endpoint values,
the coordinates $(K_t)_{i_-<t<i_+}$ are uniform among all nondecreasing
integer sequences satisfying
\[
 K_{i_-}=u,\qquad K_{i_+}=v,\qquad
 K_t\ge\ell_t\quad(i_-<t<i_+).
\]
Indeed, once $u$ and $v$ are fixed, the numbers of possible height
sequences before $i_-$ and after $i_+$ do not depend on the coordinates
strictly between them.

Admissibility gives $u\ge\ell_{i_-}$, $v\ge\ell_{i_+}$, and $u\le v$.
Apply \cref{lem:q122-monotonicity} first to raise the right endpoint from
$\ell_{i_+}$ to $v$ and then to raise the left endpoint from
$\ell_{i_-}$ to $u$.  Neither operation can decrease the expected
midpoint.  Therefore
\[
 \E[K_i\mid K_{i_-}=u,K_{i_+}=v]\ge \E_{\ell}[K_i],
\]
where $\E_{\ell}$ is expectation for the uniform sequence with endpoints
$K_{i_-}=\ell_{i_-}$, $K_{i_+}=\ell_{i_+}$ and lower bounds
$(\ell_t)_{t=i_-}^{i_+}$.  By \eqref{eq:q122-chord}, these lower bounds
lie above the affine lower bounds $\bar\ell_t-2$.  Raising lower bounds
cannot decrease the expected midpoint, again by
\cref{lem:q122-monotonicity}.  Consequently,
\[
 \E_{\ell}[K_i]\ge \E_{\mathrm{aff}}[K_i],
\]
where $\E_{\mathrm{aff}}$ denotes expectation for the uniform sequence
with the same endpoints and only the conditions
$K_t\ge\bar\ell_t-2$.

Under this last law, define the increments
\[
 D_j=K_{i_-+j}-K_{i_-+j-1},\qquad 1\le j\le m.
\]
They are nonnegative integers with sum $V$, and every such $m$-tuple
satisfying the lower-bound condition has equal probability.  Moreover,
\[
 K_{i_-+t}=\ell_{i_-}+\sum_{j=1}^tD_j,
 \qquad
 \bar\ell_{i_-+t}=\ell_{i_-}+\frac{tV}{m}.
\]
Thus the affine lower-bound condition is exactly
\[
 S_t:=2+\sum_{j=1}^t\left(D_j-\frac Vm\right)\ge0
 \qquad(1\le t\le m).
\]
Apply \cref{lem:q122-excursion} with
$\kappa=2$, $a=1/(2r)$, and $b=10r$.  Equation
\eqref{eq:q122-slope} verifies its hypothesis and gives
\[
 \E_{\mathrm{aff}}[S_R]\ge\gamma_r\sqrt m
 =\gamma_r\sqrt{2R}.
\]
At the midpoint $i_-+R=i$,
\[
 S_R=2+K_i-\bar\ell_i.
\]
Using $\bar\ell_i-\ell_i\ge-2$ from \eqref{eq:q122-chord}, we conclude,
for every possible pair $(u,v)$, that
\[
 \E[K_i-\ell_i\mid K_{i_-}=u,K_{i_+}=v]
 \ge\gamma_r\sqrt{2R}-4.
\]
Averaging over $(u,v)$ proves \eqref{eq:q122-middle-clearance}.  Since
$R\to\infty$, the right-hand side is at least $r$ for all sufficiently
large $N$, depending on $r$.  Returning to
\eqref{eq:q122-Gidentity} proves $G_i\ge r$.
\end{proof}

\end{document}